\documentclass{amsart}
\usepackage[all]{xy}
\newtheorem{theorem}{Theorem}[section]
\newtheorem{lemma}[theorem]{Lemma}
\newtheorem{corollary}[theorem]{Corollary}
\newtheorem{proposition}[theorem]{Proposition}
\newtheorem{conjecture}[theorem]{Conjecture}

\theoremstyle{definition}
\newtheorem{definition}[theorem]{Definition}

\theoremstyle{remark}
\newtheorem{remark}[theorem]{\bf Remark}

\newcommand{\Q}{\mathbb Q}

\newcommand{\Z}{\mathbb Z}

\newcommand{\N}{\mathbb N}
\newcommand{\F}{\mathbb F}

\begin{document}

\title[Stickelberger splitting in the $K$--theory of number fields]{{On the Stickelberger splitting map in the $K$--theory of number fields}}
%\pagestyle{headings}

\author[G. Banaszak]{Grzegorz Banaszak*}
\address{Department of Mathematics and Computer Science, Adam Mickiewicz University,
Pozna\'{n} 61614, Poland}
\email{banaszak@amu.edu.pl}

\author[C. Popescu]{Cristian Popescu**}
\address{Department of Mathematics, University of California, San Diego, La Jolla, CA 92093, USA}
\email{cpopescu@math.ucsd.edu}

%    General info
\date{}

\thanks{*The first author was partially supported by
a research grant of the Polish Ministry of Science
and Education. **The second author was partially supported by NSF grants
DMS-901447 and DMS-0600905}
\begin{abstract}
{The Stickelberger splitting map in the case of abelian extensions
$F / \Q$ was defined in [Ba1, Chap. IV]. The construction used
Stickelebrger's theorem. For abelian extensions $F / K$ with an
arbitrary totally real base field $K$ the construction of \cite{Ba1}
cannot be generalized since Brumer's conjecture (the analogue of
Stickelberger's theorem) is not proved yet at that level of generality.
In this paper, we construct a general Stickelberger splitting map
under the assumption that the first Stickelberger elements
annihilate the Quillen $K$--groups groups $K_2 ({\mathcal
O}_{F_{l^k}})$ for the Iwasawa tower $F_{l^k} := F(\mu_{l^k})$,
for $k \geq 1.$ The results of [Po] give examples of CM abelian
extensions $F/K$ of general totally real base-fields $K$ for which
the first Stickelberger elements annihilate $K_2 ({\mathcal
O}_{F_{l^k}})_l$ for all $k \geq 1$, while this is proved in full generality in
[GP], under the assumption that the Iwasawa $\mu$--invariant $\mu_{F,l}$ vanishes.
As a consequence, our Stickelberger splitting map
leads to annihilation results as predicted by the original
Coates-Sinnott conjecture for the subgroups $div(K_{2n}(F)_l)$ of $K_{2n}(O_F)_l$ consisting
of all the $l$--divisible elements
in the even Quillen $K$-groups of
$F$, for all odd primes $l$ and all $n$. } In \S6, we construct
a Stickelberger splitting map for \'etale $K$--theory. Finally, we construct both the Quillen
and \'etale Stickelberger splitting maps under the more general assumption
that for some arbitrary but fixed natural number $m>0$, the corresponding $m$--th Stickelberger elements
annihilate $K_{2m} ({\mathcal
O}_{F_k})_l$ (respectively $K^{et}_{2m} ({\mathcal
O}_{F_k})_l$), for all $k$.
\end{abstract}

\maketitle

\section{Introduction}

\noindent Let $F / K$ be an abelian CM extension of a totally real
number field $K.$ Let ${\bf f}$ be the conductor of $F / K$ and
let $K_{{\bf f}} / K$ be the ray--class field extension with
conductor ${\bf f}.$ Let $G_{\bf f} := G(K_{{\bf f}} / K).$ For all $n\in\Z_{\geq 0}$, Coates
\cite{C} defined higher Stickelberger elements $\Theta_{n} ({\bf
b}, {\bf f}) \in \Q [G(F / K)]$, for integral ideals $\bf b$ of
$K$ coprime to $\bf f$. Deligne and Ribet proved that $\Theta_{n}
({\bf b}, {\bf f}) \in \Z[G(F / K)]$. (See section 2 below for the
detailed discussion of the Stickelberger elements and their basic
properties.) In 1974, Coates and Sinnott \cite{CS} formulated the
following conjecture.
\begin{conjecture}\label{Conj. Coates-Sinnott} \, (Coates-Sinnott)
$\Theta_{n} ({\bf b}, {\bf f})$
annihilates $K_{2n} ({\mathcal O}_F)$ for each $n \geq 1.$
\end{conjecture}

\noindent This should be viewed as a higher analogue of the classical conjecture of Brumer.

\begin{conjecture}\label{Conj. Brumer}\, (Brumer) $\Theta_{0} ({\bf b}, {\bf f})$
annihilates $K_{0} ({\mathcal O}_F)_{\rm tors}=Cl(O_F).$
\end{conjecture}

\noindent Coates and Sinnott \cite{CS} proved that for the base field $K =
\Q$ the element $\Theta_{1} ({\bf b}, {\bf f})$ annihilates $K_{2}
({\mathcal O}_F)$ for $F/\Q$ abelian and ${\bf b}$ coprime to the
order of $K_{2} ({\mathcal O}_F).$ Moreover, they proved that
$\Theta_{n} ({\bf b}, {\bf f})$ annihilates the $l$--adic \'etale cohomology groups $K_{2n}^{et}
({\mathcal O}_F[1/l])$ for any odd prime $l$, any odd $n$ and
$F/\Q$ abelian CM extension. One of the ingredients used in the proof is the fact
that Brumer's conjecture holds true if $K=\Bbb Q$. This is the classical theorem of Stickelberger.
The passage from annihilation of
\'etale cohomology to that of $K$--theory in the case $n=1$ was possible due to the following theorem (see \cite{Ta2}, \cite{Co1} and \cite{Co2}.)

\begin{theorem}\label{Tate's Theorem} (Tate \cite{Ta2}) The $l$--adic Chern map gives a canonical isomorphism
$${K_{2} ({\mathcal O}_L) \otimes \Z_l
\stackrel{\cong}{\longrightarrow}
 K_{2}^{et} ({\mathcal O}_L[1/l])},
$$
for any number field $L$ and any odd prime $l$.
\end{theorem}

The following conjecture (generalizing Tate's theorem) is closely
related to that of Coates and Sinnott.

\begin{conjecture}\label{Conj. Quillen-Lichtenbaum} \, (Quillen-Lichtenbaum)
For any number field $L$ any $m \geq 1$ and any odd prime $l$ there is a natural
$l$--adic Chern map isomorphism
\begin{eqnarray}{K_{m} ({\mathcal O}_L) \otimes \Z_l
\stackrel{\cong}{\longrightarrow}
 K_{m}^{et} ({\mathcal O}_L[1/l])}
\end{eqnarray}
\end{conjecture}

\noindent If the Quillen-Lichtenbaum conjecture is proved,
then the $l$--primary
part of the Coates-Sinnott conjecture is established for $F/\Q$
totally real abelian, $l$ odd and $n$ odd, via the results of \cite{CS}
mentioned above. There is hope that
recent work of Suslin, Voyevodsky, Rost,
Friedlander, Morel, Levine, Weibel and others will lead to a proof of
the Quillen-Lichtenbaum conjecture.
\medskip

A different approach towards the Coates-Sinnott conjecture was
taken upon in \cite{Ba1}, in the case $K = \Q.$ Namely, in Chap.
IV loc. cit., the first author constructed the Stickelberger
splitting map $\Lambda$ of the boundary map $\partial_F$ in the
Quillen localization sequence
$$ 0 \stackrel{}{\longrightarrow} K_{2n} ({\mathcal O}_{F})_l
\stackrel{}{\longrightarrow} K_{2n} (F)_l
{{\stackrel{\partial_F}{\longrightarrow}} \atop
{{\stackrel{\Lambda}{\longleftarrow}}}} \bigoplus_{v} K_{2n-1}
(k_v)_l \stackrel{}{\longrightarrow} 0\,, $$ such that $\Lambda
\circ \partial_F$ is the multiplication by $\Theta_{n} ({\bf b},
{\bf f}).$ This property implies that $\Theta_{n} ({\bf b}, {\bf
f})$ annihilates the group $div (K_{2n} (F)_l)$ of divisible elements
in $K_{2n}(F)_l$, which is contained
in $K_{2n} ({\mathcal O}_F)_l$ (obvious from the exact sequence above and
the finiteness of  $K_{2n-1}
(k_v)_l$, for all $v$.) The construction of $\Lambda$ was
done without appealing to \'etale cohomology and the
Quillen-Lichtenbaum conjecture. However, the construction in loc.
cit. was based on the fact that Brumer's Conjecture is known to
hold for abelian extensions of $\Q$ (Stickelberger's
theorem). Since Brumer's conjecture is not yet proved
over arbitrary totally real base fields, the construction of
$\Lambda$ in loc. cit. cannot be generalized.

In this paper, we take yet another approach to the construction of the
map $\Lambda$ for arbitrary totally real base fields. Namely, we
work under the assumption that the Stickelberger elements
$\Theta_{1} ({\bf b}, {\bf f}_k)$ annihilate $K_{2} ({\mathcal
O}_{F_k})_l$ for each $k$, where $F_k := F(\mu_{l^k})$ and ${\bf
f}_k$ is the conductor of $F_k/F$. At this level of generality and
under this assumption, the construction of $\Lambda$ uses
different techniques and is more elaborate then the one in \cite{Ba1}.
In the construction, we need to operate at all levels of the
Iwasawa tower $F_k,$ for $k \geq 1$ and for every prime $v.$ In
\cite{Ba1}, the construction at each prime $v$ used only a certain level
of the Iwasawa tower. However, our efforts pay off. Even in the
particular case $K = \Q,$ this new Stickelberger splitting map
construction improves upon the results in \cite{Ba1}, where the case $l
| n$ was only settled up to a factor of $l^{v_l (n).}$ This factor
is completely eliminated in this paper.
Moreover, it was shown in
[Po] that for many examples of extensions of arbitrary totally
real base fields the Stickelberger element $\Theta_{1} ({\bf b},
{\bf f}_k)$ indeed annihilates $K_{2} ({\mathcal O}_{F_k})_l.$
Also, Greither-Popescu have recently showed in [GP] that under
the hypothesis that the Iwasawa $\mu$--invariant associated to $F$ and $\ell$
vanishes (a classical conjecture of Iwasawa), then $\Theta_{n}({\bf b},
{\bf f}_k)$ annihilates $K_{2n}^{et}(O_F[1/\ell])$, for all odd primes
$\ell$ and all odd $n$. In particular, if combined with Tate's theorem, this result
implies that
$\Theta_{1} ({\bf b},
{\bf f}_k)$ annihilates $K_{2} ({\mathcal O}_{F_k})_l$, for all $k$, under the above hypothesis
for $F$.
This way, we get annihilation results of the group $div
(K_{2n} (F)_l)$ for extensions $F/K$ with arbitrary totally real
base field $K$ (see Theorems \ref{Theorem 5.4} and \ref{Theorem 5.8}.)

In \S6, we describe briefly  the construction of the
Stickelberger splitting $\Lambda^{et}$ for the \' etale
$K$-theory which is a direct analogue of the map $\Lambda.$
The \'etale Stickelberger splitting map $\Lambda^{et}$ has similar properties and applications
as  $\Lambda$. Finally, in \S7, we construct both $\Lambda$ and $\Lambda^{et}$ under the more general assumption
that for some arbitrary but fixed natural number $m>0$, the $m$--th Stickelberger elements
$\Theta_{m} ({\bf b}, {\bf f}_k)$ annihilate $K_{2m} ({\mathcal
O}_{F_k})_l$ (respectively $K^{et}_{2m} ({\mathcal
O}_{F_k})_l$) for each $k$.
\medskip

We conclude this introduction with a few paragraphs showing that the groups
of divisible elements in the $K$--theory
of number fields lie at the heart of several important conjectures in number theory,
trying to justify
this way our efforts to understand their Galois-module structure in terms of special values of global $L$--functions.
In 1988, Warren Sinnott pointed out to the first author that
Stickelberger's Theorem for an abelian extension $F/\Q$
 or, more generally, Brumer's conjecture for a CM extension $F/K$
of a totally real number field  $K$ is equivalent to the existence of
a Stickelberger splitting map $\Lambda$ in the following basic exact sequence
$$
0 \stackrel{}{\longrightarrow} {\mathcal O}_{F}^{\times}
\stackrel{}{\longrightarrow} F^{\times}
{{\stackrel{\partial_F}{\longrightarrow}} \atop
{{\stackrel{\Lambda}{\longleftarrow}}}} \bigoplus_{v} \Z
\stackrel{}{\longrightarrow} Cl(O_F) \stackrel{}{\longrightarrow} 0
.$$
This means that $\Lambda$ is a group homomorphism, such that
$\partial_F \circ \Lambda$ is the multiplication by
$\Theta_{0} ({\bf b}, {\bf f}).$ Obviously, the above exact sequence
is the lower part of the Quillen localization sequence in $K$--theory, since
$K_1 ({\mathcal O}_{F}) = {\mathcal O}_{F}^{\times},$
$K_1 (F) = F^{\times},$ $K_0 (k_v) = \Z,$ $K_0 ({\mathcal O}_{F})_{tors} = Cl(O_F)$
and Quillen's $\partial_F$ is the direct sum of the valuation maps, just as above.

Further, by \cite{Ba2} p. 292 we observe that for any prime $l > 2$,
the annihilation of $div (K_{2n} (F)_l)$ by $\Theta_{n} ({\bf b}, {\bf f})$ is equivalent to
the existence of a ``splitting'' map $\Lambda$ in the following exact sequence
$$
0 \stackrel{}{\longrightarrow} K_{2n} ({\mathcal O}_{F}) [l^k]
\stackrel{}{\longrightarrow} K_{2n} (F) [l^k]
{{\stackrel{\partial_F}{\longrightarrow}} \atop
{{\stackrel{\Lambda}{\longleftarrow}}}} \bigoplus_{v} K_{2n - 1} (k_v) [l^k]
\stackrel{}{\longrightarrow} div
(K_{2n} (F)_l) \stackrel{}{\longrightarrow} 0
$$
such that $\partial_F \circ \Lambda$ is the multiplication by
$\Theta_{n} ({\bf b}, {\bf f})$, for any $k\gg 0$.
Hence, the group of divisible elements $div (K_{2n} (F)_l)$ is a direct analogue
of the $l$--primary part $Cl(O_F)_l$ of the class group.
Any two such ``splittings'' $\Lambda$ differ by a homomorphism
in ${\rm Hom} (\bigoplus_{v} K_{2n - 1} (k_v) [l^k], \,\, K_{2n} ({\mathcal O}_{F}) [l^k]).$
Moreover, the Coates-Sinnott conjecture is equivalent to the existence of a ``splitting'' $\Lambda$,
such that $\Lambda \circ \partial_F$ is the multiplication
by $\Theta_{n} ({\bf b}, {\bf f}).$ If the Coates-Sinnott conjecture holds, 
then such a ``splitting'' $\Lambda$ is unique and
satisfies the property that $\partial_F \circ \Lambda$ is equal to the multiplication
by $\Theta_{n} ({\bf b}, {\bf f})$. This is due to the fact that $div(K_{2n} (F)_l) \subset
K_{2n} ({\mathcal O}_{F})_l.$ Clearly, in the case $div(K_{2n} (F)_l) =
K_{2n} ({\mathcal O}_{F})_l$, our map $\Lambda$ also has the property that
$\Lambda \circ \partial_F$ equals multiplication
by $\Theta_{n} ({\bf b}, {\bf f}).$ Observe that if the Quillen-Lichtenbaum
conjecture holds, then by Theorem 4 in \cite{Ba2}, we have
$$div(K_{2n} (F)_l) =
K_{2n} ({\mathcal O}_{F})_l\,\, \Leftrightarrow \,\, \left\vert
\frac{\prod_{v | l} w_n (F_v)}{w_n (F)} \right\vert_{l}^{-1} = 1.$$
In particular, for $F = \Q$ and $n$ odd, we have $w_n (\Q) = w_n (\Q_l) = 2$.
Hence, according to the Quillen-Lichtenbaum conjecture, for any $l > 2$
we should have $div (K_{2n} (\Q)_l) = K_{2n} (\Z)_l.$

Now, let $A := Cl(\Z [\mu_l])_l$ and let
$A^{[i]}$ denote the eigenspace corresponding  to the $i$--th power of the
Teichmuller character $\omega\, :\, G(\Q(\mu_l) / \Q) \rightarrow (\Z/l\Z)^{\times}.$
Consider the following classical conjectures in cyclotomic field theory.
\begin{conjecture}\label{Conj. Kumer-Vandiver} \, (Kummer-Vandiver)
\begin{eqnarray}
A^{l - 1 -n} = 0 \quad \text{for all n even and} \quad 0 \leq n \leq l-1
\nonumber
\end{eqnarray}
\end{conjecture}

\begin{conjecture}\label{Conj. Iwasawa} \, (Iwasawa)
\begin{eqnarray}
A^{l - 1 -n} \quad \text{is cyclic for all $n$ odd, such that} \quad
1 \leq n \leq l-2
\nonumber
\end{eqnarray}
\end{conjecture}

\noindent We can state the Kummer-Vandiver and Iwasawa
conjectures in terms of divisible elements in $K$--theory of $\Q$ (see \cite{BG1} and \cite{BG2}):
\begin{itemize}
\item[(1)]
$A^{l - 1 -n} = 0$ $\Leftrightarrow$ $div (K_{2n} (\Q)_l) = 0, \,\,\,\, \text{ for all } n$
\text{even, with} $0 \leq n \leq l-1.$
\item[(2)]
$A^{l - 1 -n}$ \text{is cyclic} $\Leftrightarrow$
$div (K_{2n} (\Q)_l)$\text{ is cyclic, for all $n$
odd, with} $1 \leq n \leq l-2.$
\end{itemize}

Finally, we would like to point out that the groups of divisible elements discussed in this paper are also related
to the Quillen-Lichtenbaum conjecture.
Namely, by comparing the exact sequence
of \cite{Sch}, Satz 8 with the exact sequence of \cite{Ba2}, Theorem 2 we conclude
that the Quillen-Lichtenbaum conjecture
for the $K$-group $K_{2n} (F)$ (for any number field $F$ and any prime $l > 2$)
holds if and only if
$$div (K_{2n} (F)_l) = K_{2n}^{w} ({\mathcal O}_{F})_l$$
where $K_{2n}^{w} ({\mathcal O}_{F})_l$ is the wild kernel defined in
\cite{Ba2}.

\noindent
\section{Basic facts about the Stickelberger ideals}
Let $F / K$ be abelian CM extension of a totally real number field
$K.$ Let ${\bf f}$ be the conductor of $F / K$ and let $K_{{\bf
f}} / K$ be the ray class field extension corresponding to ${\bf
f}.$ Let $G_{\bf f} := G(K_{{\bf f}} / K).$ Every element of
$G_{\bf f}$ is the Frobenius morphism $\sigma_{{\bf a}}$, for some
ideal $\bf a$ of ${\mathcal O}_K$, coprime to the conductor $\bf
f$. Let $({\bf a}, F)$ denote the image of $\sigma_{{\bf a}}$ in
$G(F / K)$ via the natural surjection $G_{\bf f} \rightarrow G(F /
K).$ Choose a prime number $l.$
\medskip

\noindent With the usual notations, we let $I({\bf f}) / P_{1}
({\bf f})$ be the ray class group of fractional ideals in $K$
coprime to ${\bf f}.$ Let ${\bf a}$ and ${\bf a}^{\prime}$ be two
fractional ideals in $I({\bf f}).$ The symbol ${\bf a} \equiv {\bf
a}^{\prime}  \mod {\bf f}$ will mean that ${\bf a}$ and ${\bf
a}^{\prime}$ are in the same class modulo $P_{1} ({\bf f}).$ For
$\rm{Re} (s) > 1$ consider the partial zeta function of \cite{C},
p. 291

\begin{equation}
\zeta_{{\bf f}} ({\bf a}, s) := \sum_{{\bf c} \equiv {\bf a}  \,
\rm{mod} \,  {\bf f} } \,\, {\frac{1} {N{\bf c}^s}}\,,
\label{2.0}\end{equation} where the sum is taken over the integral
ideals $\bf c\in I(\bf f)$ and $N\bf c$ denotes the usual norm of
the integral ideal $\bf c$. \noindent The partial zeta
$\zeta_{{\bf f}} ({\bf a}, s)$ can be meromorphically continued to
the complex plane with a single pole at $s = 1.$ For $s\in\Bbb
C\setminus\{1\}$, consider the Sickelberger element of [C], p.
297,
\begin{equation}
\Theta_{s} ({\bf b}, {\bf f}) := (N {\bf b}^{s+1} - ({\bf b}, \,
F)) \sum_{\bf a}\, \zeta_{{\bf f}} ({\bf a}, -s) ({\bf a},\,
F)^{-1}\in \Bbb C[G(F/K)] \label{2.1}\end{equation} where the summation is over a
finite set $\mathcal S$ of ideals ${\bf a}$ of ${\mathcal O}_{K}$
coprime to ${\bf f}$ such that the Artin map
$$\mathcal S\longrightarrow G(K_{\bf f} /
K)\,,\quad \bf a\longrightarrow \sigma_{\bf a}$$ is bijective. The
element $\Theta_{s} ({\bf b}, {\bf f})$ can be written in the
following way

\begin{equation}
\Theta_{s} ({\bf b}, {\bf f}) :=
\sum_{\bf a}\, \Delta_{s+1} ({\bf a}, {\bf b}, {\bf f} ) ({\bf a} ,\, F)^{-1},
\label{2.7}\end{equation}
where
\begin{equation}
\Delta_{s+1} ({\bf a} , {\bf b} , {\bf f} ) :=
N {\bf b} ^{s+1} \zeta_{{\bf f}} ({\bf a} , -s) -
\zeta_{{\bf f}} ({\bf a} {\bf b} , -s).
\label{2.8}\end{equation}

Arithmetically, the Stickelberger elements $\Theta_{s} ({\bf b},
{\bf f})$ are most interesting for values $s = n \in \N \cup
\{0\}.$ If ${\bf a}, {\bf b}, {\bf f}$ are integral ideals, such
that ${\bf a} {\bf b}$ is coprime to ${\bf f}$,  then Deligne and
Ribet [DR] proved that $\Delta_{n+1} ({\bf a}, {\bf b}, {\bf f})$
are $l$-adic integers for all primes $l \not |\, N {\bf b}$ and
all $n \geq 0$. Moreover,

\begin{equation}
\Delta_{n+1} ({\bf a} , {\bf b} , {\bf f} ) \equiv N ({\bf a} {\bf
b} )^n \Delta_{1} ({\bf a}, {\bf b}, {\bf f}) \mod w_{n}(K_{\bf
f}). \label{2.9}\end{equation} As usual, if $L$ is a number field,
then $w_{n}(L)$ is the largest number $m \in \N$ such that the
Galois group $G(L (\mu_m) / L)$ has exponent dividing $n.$ Note
that $$w_{n}(L) = |H^0 (G(\overline{L}/L), \, \Q / \Z(n))|\,,$$
where $\Q /\Z (n) := \oplus_{l} \Q_l / \Z_l (n).$ By Theorem 2.4
of [C],  we have
$$\Theta_{n} ({\bf b}, {\bf f}) \in \Z[G(F/ K)],$$ whenever
${\bf b}$ is coprime to $w_{n+1} (F).$ The ideal of $\Z [G (F /
K)]$ generated by the elements $\Theta_{n} ({\bf b}, {\bf f})$,
for all integral ideals ${\bf b}$ coprime to $w_{n+1} (F)$ is
called the $n$-th Stickelberger's ideal for $F/K.$

\medskip
When $K \subset F \subset E$ is a tower of finite abelian
extensions then $Res_{E/F} \, : \, G(E/K) \rightarrow G(F/K)$
denotes the restriction map. We will also use the notation
$Res_{E/F}: \Bbb C[G(E/K)]\rightarrow \Bbb C[G(F/K)]$ for the
restriction map at the level of the corresponding group rings.
 When ${\bf f}\, | \, {\bf f}^{\prime}$ and ${\bf f}$ and ${\bf
f}^{\prime}$ are divisible by the same prime ideals of ${\mathcal
O}_K$ then,  for all ${\bf b}$ coprime to ${\bf f}$,  we have the
following equality (see \cite{C} Lemma 2.1, p. 292).

\begin{equation}
Res_{K_{{\bf f}^{\prime}}/K_{{\bf f}}} \,
\Theta_{s} ({\bf b}, {\bf f}^{\prime}) =
\Theta_{s} ({\bf b}, {\bf f}).
\label{2.2}\end{equation}
\medskip

\noindent If ${\bf l}$ is a prime ideal of $O_K$ coprime to ${\bf
f}$, then

\begin{equation}
\zeta_{{\bf f}} ({\bf a}, s) :=
\sum_{{{\bf c} \equiv {\bf a}  \, \rm{mod} \,
{\bf f}} \atop {{\bf l}\, \not| \, {\bf c}} }
\,\, {\frac{1}{N{\bf c}^s}} +
\sum_{{{\bf c} \equiv {\bf a} \, \rm{mod} \,  {\bf f}} \atop
{{\bf l}\, | \, {\bf c}}} \,\, {\frac{1}{N{\bf c}^s}}
\label{2.3}\end{equation}

\noindent
Observe that:
\begin{equation}
\sum_{{{\bf c} \equiv {\bf a}  \, \rm{mod} \,  {\bf f}} \atop
{{\bf l}\, \not| \, {\bf c}} } \,\, {\frac{1}{N{\bf c}^s}} =
\sum_{{{\bf a}^{\prime} \, \rm{mod} \,
{\bf l f} } \atop { {\bf a}^{\prime} \equiv {\bf a}
 \, \rm{mod} \,  {\bf f}} }  \sum_{{\bf c} \equiv {\bf a}^{\prime}
 \, \rm{mod} \,  {\bf l f} }
\,\, {\frac{1}{N{\bf c}^s}} =
\sum_{{{\bf a}^{\prime} \, \rm{mod} \,  {\bf l f} } \atop { {\bf a}^{\prime}
\equiv {\bf a}  \, \rm{mod} \,  {\bf f} }}
\zeta_{{\bf l f}} ({\bf a}^{\prime}, s)
\label{2.4}\end{equation}
\medskip

\noindent Let us fix a finite $\mathcal S$ of integral ideals
${\bf a}$ in $I(\bf f)$ as above. Observe that every class
corresponding to an integral ideal ${\bf a}$ modulo $P_{1} ({\bf
f})$ can be written uniquely as a class ${\bf l}{\bf
a}^{\prime\prime}$ modulo $P_{1} ({\bf f})$,  for some ${\bf
a}^{\prime\prime}$ from our set $\mathcal S$ of chosen integral
ideals. Since $\sigma_{{\bf l}} \in G(K_{\bf f} / K),$ this
establishes a one--to--one correspondence between classes ${\bf
a}$ modulo $P_{1} ({\bf f})$ and ${\bf a}^{\prime\prime}$ modulo
$P_{1} ({\bf f}).$ If ${\bf l}\, | \,{\bf c}$,  we put ${\bf c} =
{\bf l} {\bf c}^{\prime}.$ Hence, we have the following equality.

\begin{equation}
\sum_{{{\bf c} \equiv {\bf a}  \, \rm{mod} \,   {\bf f}} \atop
{{\bf l}\, | \, {\bf c}} }
\,\, {\frac{1}{N{\bf c}^s}} \,\, = \,\,
{\frac{1}{N{\bf l}^s}}\sum_{{\bf a}^{\prime\prime} \equiv {\bf c}^{\prime}
\, \rm{mod} \,
{\bf f} } \,\, {\frac{1}{N{\bf c}^{\prime \, s}}} =
{\frac{1}{N{\bf l}^{s}}}
\zeta_{{\bf f}} ({\bf a}^{\prime\prime}, s)
\label{2.5}\end{equation}

\noindent Formulas (\ref{2.3}), (\ref{2.4}) and (\ref{2.5}) lead
to the following equality:

\begin{equation}
\zeta_{{\bf f}} ({\bf a}, s) -
{\frac{1}{N{\bf l}^{s}}} \zeta_{{\bf f}} ({\bf l^{-1}} {\bf a}, s)
= \sum_{{{\bf a}^{\prime} \, \rm{mod} \,  {\bf l f} } \atop { {\bf a}^{\prime}
\equiv {\bf a}  \, \rm{mod} \,  {\bf f} }}
\zeta_{{\bf l f}} ({\bf a}^{\prime}, s).
\label{2.51}\end{equation}

For all ${\bf f}$ coprime to ${\bf l}$ and for all ${\bf b}$
coprime to ${\bf l f}$, equality (\ref{2.51}) gives:

\begin{equation}
Res_{K_{{\bf l} {\bf f}}/K_{{\bf f}}} \,\, \Theta_{s} ({\bf b}, {\bf l} {\bf f})
= (1 - ({\bf l}, \, F)^{-1} N{\bf l}^{s}) \,
\Theta_{s} ({\bf b}, {\bf f})
\label{2.52}\end{equation}
\medskip

\noindent
Indeed we easily check that:
\begin{equation}
Res_{K_{{\bf l} {\bf f}}/K_{{\bf f}}} \, (N {\bf b}^{s+1} - ({\bf b}, \, F))
\sum_{{{\bf a}^{\prime} \, \rm{mod} \,  {\bf l f} }}
\zeta_{{\bf l f}} ({\bf a}^{\prime}, - s) ({\bf a}^{\prime},\, F)^{-1} =
\nonumber\end{equation}
\begin{equation}
(N {\bf b}^{s+1} - ({\bf b}, \, F))
\sum_{{{\bf a}  \, \rm{mod} \,  {\bf f} }}
\sum_{{{\bf a}^{\prime} \, \rm{mod} \,  {\bf l f} } \atop  {{\bf a}^{\prime}
\equiv {\bf a}  \, \rm{mod} \,  {\bf f}}}
\zeta_{{\bf l f}} ({\bf a}^{\prime}, - s) ({\bf a},\, F)^{-1} =
\nonumber\end{equation}
\begin{equation}
(N {\bf b}^{s+1} - ({\bf b}, \, F))
\sum_{{{\bf a} \, \rm{mod} \,  {\bf f}}}(\zeta_{{\bf f}} ({\bf a}, - s) -
N{\bf l}^{s} \zeta_{{\bf f}} ({\bf l}^{-1} {\bf a}, - s))
({\bf a},\, F)^{-1} =
\nonumber\end{equation}
\begin{equation}
(N {\bf b}^{s+1} - ({\bf b}, \, F))
(\sum_{{{\bf a} \, \rm{mod} \,  {\bf f}}}\zeta_{{\bf f}} ({\bf a}, - s)
({\bf a},\, F)^{-1}
- ({\bf l}, \, F)^{-1} N{\bf l}^{s} \,
\zeta_{{\bf f}} ({\bf l}^{-1} {\bf a}, - s) ({\bf l}^{-1} {\bf a},\, F)^{-1}) =
\nonumber
\end{equation}
\begin{equation}
(1 - ({\bf l}, \, F)^{-1} N{\bf l}^{s})
(N {\bf b}^{s+1} - ({\bf b}, \, F))
\sum_{{{\bf a} \, \rm{mod} \,  {\bf f}}}\zeta_{{\bf f}} ({\bf a}, - s)({\bf a},\, F)^{-1}
\nonumber\end{equation}

\begin{lemma}\label{Lemma 2.1}
Let ${\bf f}\, | \, {\bf f}^{\prime}$ be ideals of ${\mathcal
O}_K$ coprime to ${\bf b}.$ Then we have the following equality.
\begin{equation}
Res_{K_{{\bf f}^{\prime}}/K_{{\bf f}}} \,\, \Theta_{s} ({\bf b}, {\bf f}^{\prime})
= \bigl( \, \prod_{{{\bf l} \, \not \, | \, {\bf f}} \atop
{{\bf l} \, | \, {\bf f}^{\prime}}} \, (1 - ({\bf l}, \, F)^{-1} N{\bf l}^{s}) \, \bigr) \,\,
\Theta_{s} ({\bf b}, {\bf f})
\label{2.53}\end{equation}
\end{lemma}
\begin{proof}
The lemma follows from (\ref{2.2}) and (\ref{2.52}).
\end{proof}

In what follows, for any given abelian extension $F/K$, we
consider the field extensions $F (\mu_{l^k}) / K$, for all $k \geq
0$ and a fixed prime $l$. We let ${\bf f}_k$ be the conductor of
the abelian extension $F (\mu_{l^k}) / K.$ We suppress from the
notation the explicit dependence of ${\bf f}_k$ on $l$ since the
prime $l$ is chosen once and for all in this paper.

\noindent
\section{Basic facts about algebraic $K$-theory}

\subsection{The Bockstein sequence and the Bott element}

\noindent For a ring $R$ we consider the Quillen $K$-groups
$$K_m
(R) := \pi_{m} (\Omega B Q P (R)) := [S^m, \, \Omega B Q P (R)]$$
(see \cite{Q1}) and the $K$-groups with coefficients
$$K_m (R, \,
\Z/l^k) := \pi_{m} (\Omega B Q P (R), \, \Z/l^k) := [M^{m}_{l^k},
\, \Omega B Q P (R)]$$ defined by Browder and Karoubi \cite{Br}.
Quillen's $K$--groups  can also be computed using Quillen's plus
construction as $K_n (R) := \pi_{n} (BGL (R)^{+}).$  Any unital
homomorphism of rings $\phi\, :\, R \rightarrow R^{\prime}$
induces natural homomorphisms on $K$-groups
$$\phi_{R \backslash R^{\prime}}\, :\, K_{m} (R, \, \diamondsuit)
\stackrel{}{\longrightarrow}
K_m (R^{\prime}, \, \diamondsuit)$$
where $K_{m} (R, \, \diamondsuit)$ denotes either $K_{m} (R)$ or
$K_{m} (R, \, \Z/l^k).$
\bigskip

\noindent Quillen $K$-theory and $K$-theory with coefficients admit
product structures:
$$K_n (R, \, \diamondsuit) \times K_m (R, \, \diamondsuit)
\stackrel{\ast}{\longrightarrow} K_{m + n} (R, \, \diamondsuit)$$
These induce graded ring structures on the groups $\bigoplus_{n
\geq 0} K_n (R, \, \diamondsuit).$
\medskip

\noindent For a topological space $X$, there is a Bockstein exact
sequence
$$ \stackrel{}{\longrightarrow} \pi_{m+1} (X, \, \Z/l^k)
\stackrel{b}{\longrightarrow} \pi_m (X)
\stackrel{l^k}{\longrightarrow} \pi_m (X)
\stackrel{}{\longrightarrow} \pi_m (X, \, \Z/l^k)
\stackrel{}{\longrightarrow} $$ In particular, if we take $X :=
\Omega B Q P (R))$, we get the Bokstein exact sequence in
$K$-theory:

$$ \stackrel{}{\longrightarrow} K_{m+1} (R, \, \Z/l^k)
\stackrel{b}{\longrightarrow} K_m (R) \stackrel{l^k}{\longrightarrow}
K_m (R) \stackrel{}{\longrightarrow} K_m (R, \, \Z/l^k)
\stackrel{}{\longrightarrow} $$
\medskip

\noindent
For any group $G$ we have $BG = K (G, 1).$ Hence
$$
\pi_{n} (BG) \,\, = \,\,
\left\{
\begin{array}{lll}
G&\rm{if}&n = 1\\
0&\rm{if}&n > 1\\
\end{array}\right.
$$
Consequently, for a commutative group $G$ and $X := BG$ the
Bockstein map $b$ gives an isomorphism $b\, : \, \pi_{2} (BG, \,
\Z/l^k) \stackrel{\cong}{\longrightarrow} G [l^k].$
\medskip

\noindent
For a commutative ring with identity $R$ we have
$GL_1 (R) = R^{\times}.$ Assume that $\mu_{l^k} \subset R^{\times}$. Then
$R^{\times} [l^k] = \mu_{l^k}.$
Let $\beta$ denote the natural composition of maps:
$$\mu_{l^k}
\stackrel{b^{-1}}{\longrightarrow} \pi_{2} (BGL_1 (R); \Z/l^k)
\stackrel{}{\longrightarrow} \pi_{2} (BGL (R); \Z/l^k)
\stackrel{}{\longrightarrow} \pi_{2} (BGL (R)^{+}; \Z/l^k) =
K_2 (R, \, \Z/l^k)$$
We define the Bott element $\beta_{k} :=
\beta (\xi_{l^k}) \in K_2 (R; \, \Z/l^k)$
as the image of $\xi_{l^k}$ via $\beta$, where $\xi_{l^k}$ is a fixed generator
of $\mu_{l^k}$.  We let
$$\beta_{k}^{\ast \, n} :=
\beta_{k} \ast \dots \ast \beta_{k} \in K_{2n} (R; \, \Z/l^k).$$
The Bott element $\beta_{k}$ depends of course on the ring
$R$. However, we suppress this dependence from the notation since it will be always
clear where a given Bott element lives. For example, if
$\phi\, :\, R \rightarrow R^{\prime}$ is a
homomorphism of commutative rings containing $\mu_{l^k}$,  then it is clear from
the definitions that the map
$$\phi_{R \backslash R^{\prime}}\, :\, K_{2} (R; \, \Z/l^k)
\stackrel{}{\longrightarrow}
K_2 (R^{\prime}, \, \Z/l^k)$$ transports the Bott element for
$R$ into the Bott element for $R^{\prime}$. By a slight abuse of notation, this will be written
as  $\phi_{R \backslash R^{\prime}} (\beta_{k}) = \beta_{k}.$
\medskip

Dwyer and Fiedlander [DF] constructed \' etale K-theory and proved that for any
commutative, noetherian  $\Z[1/l]$-algebra $R$ there are natural
ring homomorphisms for all $l > 2:$
\begin{equation}
K_{\ast} (R) \,\,
{\stackrel{}{\longrightarrow}} \,\,
K_{\ast}^{et} (R)
\label{infinitecoefDFmap}
\end{equation}
\begin{equation}
K_{\ast} (R; \Z/l^k ) \,\,
{\stackrel{}{\longrightarrow}} \,\,
K_{\ast}^{et} (R; \Z/l^k)
\label{finitecoefDFmap}
\end{equation}
If $R$ has finite $\Z/l$-cohomological dimension  then there are Atiyach-Hirzebruch
type spectral sequences [DF] Propositions 5.1, 5.2:
\begin{equation}
E_{2}^{p, -q} = H^p(R ; \Z_l (q/2)) \Rightarrow K_{q - p}^{et} (R).
\label{infinitecoefDFspecseq}
\end{equation}
\begin{equation}
E_{2}^{p, -q} = H^p(R ; \Z/l^k (q/2)) \Rightarrow K_{q - p}^{et} (R; \Z/l^k).
\label{finitecoefDFspecseq}
\end{equation}

\subsection{$K$-theory of finite fields}
Let $\F_q$ be the finite field with $q$ elements. Quillen [Q3]
has shown that:

$$
K_n (\F_q) \,\, = \,\,
\left\{
\begin{array}{lll}
\Z &\rm{if}&n = 0\\
0&\rm{if}&n = 2m \quad \text{and} \quad m > 0\\
\Z/(q^m - 1)\Z & \rm{if} & n = 2m -1 \quad \text{and} \quad m > 0\\
\end{array}\right.
$$

\medskip

\noindent
It was proven by Quillen [Q3] pp. 583-585 that for the finite field
extension $i \, : \, \F_q \rightarrow \F_{q^f}$ and all $n \geq 1$
the natural map:
$$i \, : \, K_{2n-1} (\F_q) \rightarrow K_{2n-1} (\F_{q^f})$$
is injective and the transfer map:
$$N \, : \, K_{2n-1} (\F_{q^f}) \rightarrow K_{2n-1} (\F_{q})$$
is surjective, where we simply write $i$ instead of
$i_{\F_q \backslash \F_{q^f}}$ and $N$ instead of $Tr_{\F_{q^f}/\F_q}.$
Moreover, Quillen [Q3] pp. 583-585
proved that
$$K_{2n-1} (\F_q) \cong K_{2n-1} (\F_{q^f})^{G(\F_{q^f} / \F_{q})}$$
and that the Frobenius automorphism $Fr_{q}$
(the canonical generator of $G(\F_{q^f} / \F_{q})$) acts on $K_{2n-1} (\F_{q^f})$
via multiplication by $q^n.$  Observe that
$$i \circ N = \sum_{i = 0}^{f-1} Fr_{q}^{i}.$$
Hence
$$\text{Ker}\, N =  \, K_{2n-1} (\F_{q^f})^{Fr_{q} - Id}\,  =
 \, K_{2n-1} (\F_{q^f})^{q^n - 1}$$
because $\text{Ker} \, N$ is the kernel of multiplication
by $\sum_{i = 0}^{f-1} \, q^{ni} = \frac {{q^{nf} - 1}}{{q^n - 1}}$
in the cyclic group $K_{2n-1} (\F_{q^f}).$
In particular, this shows that the norm map $N$ induces the following
isomorphism
$$K_{2n-1} (\F_{q^f})_{G(\F_{q^f} / \F_{q})} \cong K_{2n-1} (\F_q). $$

\medskip

\noindent
By the Bockstein exact sequence and Quillen's results above, we observe that
$$K_{2n} (\F_q,\, \Z/l^k) \stackrel{b}{\longrightarrow} K_{2n -1} (\F_q)[l^k]$$
is an isomorphism. Hence, $K_{2n} (\F_q,\, \Z/l^k)$ is a cyclic group.
Let us assume that $\mu_{l^k} \subset \F_{q}^{\times}$ (i.e. $l^k\mid q-1$.)
In this case, Browder \cite{Br} proved that the element
$\beta_{k}^{\ast \, n}$ is a generator of $K_{2n} (\F_q,\, \Z/l^k)$.
Dwyer and Friedlander proved that there is a natural isomorphism of
graded rings:
$$K_{\ast} (\F_q,\, \Z/l^k) \stackrel{\cong}{\longrightarrow}
K_{\ast}^{et} (\F_q, \, \Z/l^k ).$$
Assume $l^k \, | \, q-1$ and (by abuse of notation) let $\beta_{k}$ also denote the image of the Bott
element via the natural isomorphism:
$$K_{2} (\F_q,\, \Z/l^k) \stackrel{\cong}{\longrightarrow}
K_{2}^{et} (\F_q, \, \Z/l^k ).$$
Then by [DF] Theorem 5.6 multiplication with
$\beta_{k}$ induces isomorphisms:
$$ \times \, \beta_{k}  \, : \, K_{i}^{et} (\F_q,\, \Z/l^k)
\stackrel{\cong}{\longrightarrow} K_{i+2}^{et} (\F_q, \, \Z/l^k ),$$
$$ \times \,  \beta_{k}  \, : \, K_{i} (\F_q,\, \Z/l^k)
\stackrel{\cong}{\longrightarrow} K_{i+2} (\F_q, \, \Z/l^k ),$$
In particular,  if  $l^k \, | \, q-1$ and $\alpha \in
K_{1} (\F_q,\, \Z/l^k) = K_{1} (\F_q) /l^k$ is a generator, then
the element $\alpha \ast \beta_{k}^{\ast \, n-1}$ is a generator of the
cyclic group $K_{2n-1} (\F_q,\, \Z/l^k).$

\subsection{$K$-theory of number fields and rings of integers}
Let $F$ be a number field, let ${\mathcal O}_{F}$ be its
ring of integers and let $k_v$ be the residue field
for a prime $v$ of ${\mathcal O}_{F}.$
For a finite set of primes
$S$ of ${\mathcal O}_{F}$ the ring of $S$-integers is denoted
${\mathcal O}_{F, S}.$
\bigskip

\noindent
Quillen \cite{Q2} proved that $K_{n}({\mathcal O}_{F})$ is a finitely generated
group for every $n \geq 0.$
Borel computed the ranks of the groups $K_{n}({\mathcal O}_{F})$ as follows:
$$
 K_n ({\mathcal O}_{F}) \otimes_{\Z} \Q \,\, = \,\,
\left\{
\begin{array}{lll}
\Q   & \text{if} &  n=0\\
\Q^{r_{1} + r_{2} - 1} & \text{if} & n = 1\\
0 & \text{if} & n = 2m \quad \text{and} \quad n > 0\\
\Q^{r_1 + r_2}   & \text{if} & n \equiv 1 \mod 4 \quad
\text{and} \quad n \not= 1 \\
\Q^{r_2}  & \text{if} & n \equiv 3 \mod 4 \\
\end{array}\right.
$$

\noindent
We have the following localization
exact sequences in Quillen $K$-theory and $K$-theory with coefficients \cite{Q1}.

$$\stackrel{}{\longrightarrow} K_{m} ({\mathcal O}_{F}, \, \diamondsuit)
\stackrel{}{\longrightarrow} K_m (F, \, \diamondsuit)
\stackrel{\partial_F}{\longrightarrow}
\bigoplus_{v} K_{m-1} (k_v, \, \diamondsuit)
\stackrel{}{\longrightarrow} K_{m-1} ({\mathcal O}_{F}, \, \diamondsuit)
\stackrel{}{\longrightarrow}$$
\medskip

\noindent
Let $E/F$ be a finite extension. The natural maps in $K$-theory
induced by the embedding $i\, :\, F \rightarrow E$ and $\sigma\, :\,
E \rightarrow E,$ for $\sigma \in G(E/F)$, will be denoted
for simplicity by $i \, :\, K_{m} (F, \, \diamondsuit) \stackrel{}{\longrightarrow}
K_m (E, \, \diamondsuit)$ and $\sigma \, : \,
K_{m} (E, \, \diamondsuit) \stackrel{}{\longrightarrow}
K_m (E, \, \diamondsuit).$ Observe that  $i := i_{F \backslash  E}$ and
$\sigma := \sigma_{E \backslash E}$, according to the notation in
section 3.1.
\medskip

\noindent
In addition to the natural maps
$i,$ $\sigma,$ $\partial_F,$ $\partial_E,$ and product structures $\ast$ for
$K$-theory of $F$ and $E$ introduced above, we have (see [Q]) the transfer map

$$Tr_{E/F}\, :\, K_{m} (E, \, \diamondsuit) \stackrel{}{\longrightarrow}
K_m (F, \, \diamondsuit)$$
and the reduction map
$$r_v\, :\, K_{m} ({\mathcal O}_{F, S}, \, \diamondsuit)
\stackrel{}{\longrightarrow} K_m (k_v, \, \diamondsuit)$$
for any prime $v \notin S.$
\medskip

\noindent
The maps discussed above enjoy many compatibility properties.
For example,  $\sigma$ is naturally compatible with $i,$
$\partial_F,$ $\partial_E,$ product structure $\ast,$
$Tr_{E/F}$ and $r_w$ and $r_v.$ See e.g. \cite{Ba1} for explanations
of some of these compatibility properties.
Let us mention below two nontrivial compatibility properties.
By the result of Gillet \cite{Gi}, we have the
following commutative diagrams in Quillen $K$-theory and $K$-theory with
coefficients:

\begin{displaymath}\xymatrix{
K_{m} (F, \, \diamondsuit) \times  K_{n} ({\mathcal O}_{F}, \, \diamondsuit)
\ar@<0.1ex>[d]^{\partial \times id}  \ar[r]^{\ast} &  K_{m+n} (F, \,
\diamondsuit)
\ar@<0.1ex>[d]^{\partial}\\
\bigoplus_{v} \, K_{m-1} (k_v, \, \diamondsuit) \times
K_{n} ({\mathcal O}_{F}, \, \diamondsuit) \quad\quad
\ar[r]^{\qquad\ast} & \quad\quad \bigoplus_{v} \, K_{m + n -1} (k_v, \,
\diamondsuit)
}\label{diagram 2.4}
\end{displaymath}
\medskip

\noindent
Let $E/F$ be a finite extension unramified over a prime
$v$ of ${\mathcal O}_{F}.$  Let $w$ be a prime of
${\mathcal O}_{E}$ over $v.$ From now on, we will write
$N_{w/v} \, := \, Tr_{k_w / k_v}.$
The following diagram shows the compatibility of transfer
with the boundary map in localization sequences for Quillen $K$-theory
and $K$-theory with coefficients.

\begin{displaymath}
\xymatrix{
K_{m} (E, \, \diamondsuit) \quad \ar@<0.1ex>[d]^{Tr_{E/F}}
\ar[r]^{\bigoplus_{v} \bigoplus_{w\, | \,v} \partial_{w}} &
\quad\quad\quad\quad\quad
\bigoplus_{v} \bigoplus_{w\, | \,v} K_{m-1} (k_w, \, \diamondsuit)
\ar@<0.1ex>[d]^{\bigoplus_{v} \bigoplus_{w\, | \,v} N_{w/v}}\\
K_{m} (F, \, \diamondsuit) \quad
\ar[r]^{\bigoplus_{v}  \partial_{v}} &
\quad \bigoplus_{v} \, K_{m -1} (k_v, \, \diamondsuit)}
\label{diagram 2.41}
\end{displaymath}
\medskip

\noindent
Observe that $\partial_{E} = \bigoplus_{v} \bigoplus_{w\, | \,v} \partial_{w}$
and $\partial_{F} = \bigoplus_{v} \partial_{v}.$

\medskip

\noindent
\section{Construction of the map $\Lambda$}

Throughout the rest of the paper we assume that $\Theta_{1} ({\bf b}, {\bf f}_k)$ annihilates
$K_{2} ({\mathcal O}_{F_{l^k}})$ for all $k \geq 0.$
For a prime $v$ of ${\mathcal O}_{F}$, let
$k_v$ be its residue field and $q_v$ the cardinality of $k_v$ . Similarly, for any prime
$w$ of ${\mathcal O}_{F_{l^k}}$, we let $k_w$ be its residue field.
We put $E := F_{l^k}.$
If $v \not\,\mid l,$ we observe that $k_w = k_v (\xi_{l^k})$, since the corresponding
local field extension $E_{w}/F_{v}$ is unramified.
For any finite set $S$ of primes in $O_F$ and any $k\geq 0$, there is an exact sequence [Q].

$$
0 \stackrel{}{\longrightarrow} K_{2} ({\mathcal O}_{F_{l^k}})
\stackrel{}{\longrightarrow} K_{2} ({\mathcal O}_{F_{l^k, \, S}})
\stackrel{\partial}{\longrightarrow} \bigoplus_{v \in S} \bigoplus_{w | v}
K_{1} (k_w) \stackrel{}{\longrightarrow} 0
$$

Let $\xi_{w, k} \in K_{1} (k_w)_l$ be a generator of the $l$-torsion part of $K_1(k_w)$.
Pick an element $x_{w, k} \in K_{2} ({\mathcal O}_{F_{l^k}, S})_l$ such that
${\partial} (x_{w, k}) = \xi_{w, k}.$ Observe that
$x_{w, k}^{\Theta_{1} ({\bf b}, {\bf f}_k)}$ does not depend on the choice of $x_{w, k}$
since we assumed that $\Theta_{1} ({\bf b}, {\bf f}_k)$ annihilates
$K_{2} ({\mathcal O}_{F_{l^k}}).$ Observe that if $\text{ord}(\xi_{w, k}) = l^a$, then
$x_{w, k}^{l^a} \in K_{2} ({\mathcal O}_{F_{l^k}}).$ Hence,
$(x_{w, k}^{\Theta_{1} ({\bf b}, {\bf f}_k)})^{l^a} =
(x_{w, k}^{l^a})^{\Theta_{1} ({\bf b}, {\bf f}_k)} = 0.$
Consequently,  there is a well defined map:

$$\Lambda_{1}\, : \, \bigoplus_{v \in S} \bigoplus_{w | v} K_{1} (k_w)_l
\stackrel{}{\longrightarrow} K_{2} ({\mathcal O}_{F_{l^k, \, S}})_l ,$$
$$\Lambda_{1} (\xi_{w, k}) \, := \, x_{w, k}^{\Theta_{1} ({\bf b}, {\bf f}_k)}.$$

\begin{lemma}\label{Lemma 4.1}
The map $\Lambda_1$ satisfies
the following property
$$\partial \Lambda_{1} (\xi_{w, k}) \, :=
\, \xi_{w, k}^{\Theta_{1} ({\bf b}, {\bf f}_k)}.$$
\end{lemma}
\begin{proof}
The lemma follows immediately by compatibility of $\partial$ with
$G (E/F)$ action.
\end{proof}

Let $v$ be a prime in ${\mathcal O}_{F}$ sitting above $p \not= l$
in $\Z$. Let $S := S_v$ be the finite set primes of ${\mathcal
O}_{F}$ consisting of all the primes over $p.$ Let us fix an
$n\in\N$. Let $k (v)$ be the natural number for which $l^{k(v)}\,
|| \, q_{v}^{n} - 1.$ Observe that if $l \, | \, q_v -1$ then
$k(v) = v_{l}(q_v -1) + v_{l}(n)$ (see e.g. [Ba1, p. 336].) For $k
\geq k (v)$ and $E := F(\mu_{l^k})$ let us define elements

$$\Lambda_n (\xi_{v, k}; l^k) :=
Tr_{E/F} ( x_{w, k}^{\Theta_{1} ({\bf b}, {\bf f}_k) } \ast
\beta_{k}^{\ast \, n-1})^{N {\bf b}^{n-1}} \in K_{2n} ({\mathcal
O}_{F, S};\, \Z/l^k)).$$ From now on,  we will suppress the index
$n$ from the notation and we write $\Lambda (\xi_{v, k}; l^k)$
instead of $\Lambda_n (\xi_{v, k}; l^k).$

Let us fix a prime sitting above $v$ in each of the fields $F(\mu_{l^k})$,
such that if $k\leq k'$ and $w$ and $w'$ are the fixed primes in $E:=F(\mu_{l^k})$ and
$E':=F(\mu_{l^{k'}})$, respectively, then $w'$ sits above $w$. By the surjectivity of
the transfer maps for $K$-theory of finite fields (see the end of
section 3) we can associate to each $k$ and the chosen prime $w$ in
$E:=F(\mu_{l^k})$ a generator $\xi_{w, k}$ of $K_1(k_w)_l$, such that
$$N_{w^{\prime}/w}(\xi_{w^{\prime},
k^{\prime}})= \xi_{w, k},$$
for all $k\leq k'$, where $w$ and $w'$ are the fixed primes in $E:=F(\mu_{l^k})$ and $E':=F(\mu_{l^{k'}})$,
respectively.

Let $r_{k^{\prime}/k} \, :\,  K_{\ast} (\, .\, ;\, \Z/l^{k^{\prime}})
\rightarrow K_{\ast} (\, .\, ;\, \Z/l^k)$ be the coefficient reduction map.
Recall that we put $N_{w / v} := Tr_{k_{w}/k_v}$ and
$N_{w^{\prime} / v} := Tr_{k_{w^{\prime}}/k_v}.$

\begin{lemma}\label{Lemma 4.2} With notations as above, for every
$k \leq k^{\prime}$ we have
$$r_{k^{\prime}/k} ( N_{w^{\prime}/v}
(\xi_{w^{\prime}, k^{\prime}} \ast \beta_{k^{\prime}}^{\ast \, n-1})) =
N_{w/v} (\xi_{w, k} \ast \beta_{k}^{\ast \, n-1})$$
\end{lemma}

\begin{proof} The formula follows by the compatibility of the elements
$(\xi_{w, k})_{w}$ with respect to the norm maps, by
the compatibility of Bott elements with respect to the coefficient reduction map
$r_{k^{\prime}/k} (\beta_{k^{\prime}}) = \beta_{k},$
and by the projection formula. More precisely,
we have the following equalities:

$$r_{k^{\prime}/k} ( N_{w^{\prime}/v}(
\xi_{w^{\prime}, k^{\prime}} \, \ast \, \beta_{k^{\prime}}^{\ast \, n-1}))  =
N_{w^{\prime}/v} (r_{k^{\prime}/k} (
\xi_{w^{\prime}, k^{\prime}} \, \ast \, \beta_{k^{\prime}}^{\ast \, n-1})) =
N_{w^{\prime}/v} (
\xi_{w^{\prime}, k^{\prime}} \, \ast \, \beta_{k}^{\ast \, n-1})) =$$
$$= N_{w/v} (N_{w^{\prime}/w} (\xi_{w^{\prime}, k^{\prime}}) \ast
\beta_{k}^{\ast \, n-1})) = N_{w/v} (\xi_{w, k} \ast \beta_{k}^{\ast \, n-1}).
$$
\end{proof}

\begin{lemma}\label{Lemma 4.3}
For all $k(v) \leq k \leq k^{\prime}$, we have
$$r_{k^{\prime}/k} (\Lambda (\xi_{v, k^{\prime}};
l^{k^{\prime}})) =
\Lambda (\xi_{v, k}; l^k)$$
\end{lemma}

\begin{proof}
Consider the following commutative diagram:

$$\xymatrix{
K_{2} ({\mathcal O}_{E^{\prime}, S}) \quad \ar@<0.1ex>[d]^{Tr_{E^{\prime}/E}}
\ar[r]^{\bigoplus_{w^{\prime} \in S} \partial_{w^{\prime}}} \quad &
\quad
\bigoplus_{w^{\prime} \in S}  K_{1} (k_{w^{\prime}})
\ar@<0.1ex>[d]^{\bigoplus_{w \in S} \bigoplus_{w^{\prime} | w} N_{{w^{\prime} / w}}}\\
K_{2} ({\mathcal O}_{E, S}) \quad
\ar[r]^{\bigoplus_{w \in S} \partial_{w}} &
\quad \bigoplus_{w \in S} \, K_{1} (k_w)}
\label{diagram 2.5}$$
It follows that we have
$Tr_{E^{\prime}/E} ({x}_{w^{\prime}, k^{\prime}})^{\Theta_{1} ({\bf b}, {\bf f}_k)} =
{x}_{w, k}^{\Theta_{1} ({\bf b}, {\bf f}_k)}.$
Hence, if we use the projection formula again, we get
$$r_{k^{\prime}/k} (Tr_{E^{\prime}/F}
(x_{w^{\prime}, k^{\prime}}^{\Theta_{1} ({\bf b}, {\bf f}_{k^{\prime}})}
\ast \, \beta_{k^{\prime}}^{\ast \, n-1})^{N {\bf b}^{n-1}}) =
Tr_{E/F} (Tr_{E^{\prime}/E}
(x_{w^{\prime}, k^{\prime}}^{\Theta_{1} ({\bf b}, {\bf f}_{k^{\prime}})}
\ast \, \beta_{k}^{\ast \, n-1}))^{N {\bf b}^{n-1}} =$$
$$ = Tr_{E/F} (x_{w, k}^{\Theta_{1} ({\bf b}, {\bf f}_{k})}
\ast \, \beta_{k}^{\ast \, n-1})^{N {\bf b}^{n-1}}.$$
\end{proof}

\noindent
Let us introduce the following notation $N := \bigoplus_{v \in S} \bigoplus_{w\, | \,v} N_{w/v}.$

\begin{proposition}\label{Proposition 4.4}
For every $k \geq k(v)$ we have
$$\partial_{F} (\Lambda (\xi_{v, k}; l^k)) =
(N(\xi_{w, k} \ast \beta_{k}^{\ast \, n-1}))^{\Theta_{n} ({\bf b}, {\bf f})}$$
\end{proposition}

\begin{proof}
The proof is similar to the proofs of
[Ba1 Theorem 1, pp. 336-340] and [BG1, Proposition 2, pp. 221-222].
The diagram at the end of section 3 gives the following
commutative  diagram of $K$--groups of coefficients

$$\xymatrix{
K_{2n} ({\mathcal O}_{E, S} ;\, \Z/l^k ) \quad
\ar@<0.1ex>[d]^{Tr_{E/F}} \ar[r]^{\partial_{E}} \quad &
\quad\
\bigoplus_{v \in S} \bigoplus_{w\, | \,v} K_{2n-1} (k_w ; \, \Z/l^k )
\ar@<0.1ex>[d]^{N}\\
K_{2n} ({\mathcal O}_{F, S}  ;\, \Z/l^k) \quad
\ar[r]^{\partial_{F}} &
\quad \bigoplus_{v \in S} \, K_{2n-1} (k_v, ;\, \Z/l^k )}\,,
\label{diagram 2.6}$$
Hence we have $\partial_{F} \circ Tr_{E/F} = N \circ \partial_{E}.$
The compatibilities of some of the natural maps mentioned in section 3
which will be used next can be expressed via the following commutative diagrams,
explaining the action  of the groups $G(E/K)$ and $G(F/K)$ on the $K$--groups with coefficients
in the diagram above.

$$\xymatrix{
K_{2n-2} ({\mathcal O}_{E, S}; \, \Z/l^k ) \quad \ar@<0.1ex>[d]^{\sigma_{{\bf a}}^{-1}}
\ar[r]^{r_w} & \quad
K_{2n-2} (k_{w}; \, \Z/l^k)
\ar@<0.1ex>[d]^{\sigma_{{\bf a}}^{-1}}\\
K_{2n-2} ({\mathcal O}_{E, S}; \, \Z/l^k ) \quad
\ar[r]^{r_{w^{\sigma_{{\bf a}}^{-1}}}} & \quad
\, K_{2n-2} (k_{w^{\sigma_{{\bf a}}^{-1}}}; \, \Z/l^k)}
\label{diagram 2.7}$$
The above diagram shows that
\begin{equation}\label{reduction}  r_{w^{\sigma_{{\bf a}}^{-1}}}(\beta_{k}^{\ast \, n-1}) =
r_{w^{\sigma_{{\bf a}}^{-1}}}((\beta_{k}^{\ast \, n-1})^{ N {\bf a}^{n-1}\sigma_{{\bf a}}^{-1}})
= (r_w (\beta_{k}^{\ast \, n-1}))^{ N {\bf a}^{n-1} \sigma_{{\bf a}}^{-1}}.\end{equation}

We can write the 1st Stickelberger element as follows
\begin{equation}\label{theta-one} \Theta_{1} ({\bf b}, {\bf f}_k) =
{\sum_{{\bf a} \, {\rm{mod}} \, {\bf f}_k}}' \,\, (\sum_{{\bf c} \, {\rm{mod}} \, {\bf f}_k, \,
w^{\sigma_{{\bf c}^{-1}}} = w}\,
\Delta_{2} ({\bf a} {\bf c}, {\bf b}, {\bf f}) \sigma_{{\bf c}^{-1}}) \sigma_{{\bf a}^{-1}}\,,\end{equation}
where ${\sum'_{{\bf a} \, \rm{mod} \, {\bf f}_k}}$ denotes the sum over a maximal set $\mathcal S$
of ideal classes ${\bf a} \mod {\bf f}_k$,  such that the primes $w^{\sigma_{{\bf a}}^{-1}}$, for ${\bf a}\in\mathcal S$,  are
distinct. Since for every $m \geq 1$ we have
\begin{equation}
 \Delta_{m+1} ({\bf a}, {\bf b}, {\bf f}) \equiv
N {\bf a}^{m} \, N {\bf b}^{m} \, \Delta_{1} ({\bf a} {\bf c}, {\bf b}, {\bf f})
\mod \,\, w_{m}(K_{{\bf f}})
\label{deltacong1}\end{equation}

(see \cite{DR}),
it is clear that for all $n \geq 2$ we get:

$$\Theta_{n} ({\bf b}, {\bf f}_k) \, \equiv \,$$
$${\sum_{{\bf a} \, {\rm{mod}} \, {\bf f}_k}}' \,\, (\sum_{{\bf c} \, {\rm{mod}} \, {\bf f}_k, \,
w^{\sigma_{{\bf c}^{-1}}} = w}\,
N {\bf a}^{n-1} \, N {\bf c}^{n-1}\, N {\bf b}^{n-1} \, \Delta_{2} ({\bf a} {\bf c}, {\bf b},
{\bf f}_k)
\sigma_{{\bf c}^{-1}}) \sigma_{{\bf a}^{-1}} \,\, {\rm{mod}} \,\, w_{1}(K_{{\bf f}_k}).$$
Equalities (\ref{reduction}) and (\ref{theta-one}), the result of Gillet \cite{Gi} and
the above congruences satisfied by Stickelberger elements lead to the following equalities.

\begin{eqnarray}\nonumber
\partial_{E} ( x_{w, k}^{\Theta_{1} ({\bf b}, {\bf f}_k)} \ast
\beta_{k}^{\ast \, n-1})^{N {\bf b}^{n-1}}= \\
 \nonumber
 ={\sum_{{\bf a} \, {\rm{mod}} \, {\bf f}_k}}'
\,\, \xi_{w, k}^{\sum_{{\bf c} \, {\rm{mod}} \, {\bf f}_k, \, w^{\sigma_{{\bf c}^{-1}}} = w} \,
\Delta_{2} ({\bf a} {\bf c}, {\bf b}, {\bf f}_k) \sigma_{({\bf a} {\bf c})^{-1}}}
\ast (\beta_{k}^{{\ast \, n-1}})^{(N {\bf a} {\bf c})^{n-1} N {\bf b}^{n-1}
\sigma_{(ac)^{-1}}}= \\
\nonumber
 =(\xi_{w, k}
\ast \beta_{k}^{{\ast \, n-1}})^{
\sum'_{{\bf a} \, {\rm{mod}} \, {\bf f}_k} \sum_{{\bf c} \, {\rm{mod}} \, {\bf f}_k, \,
w^{\sigma_{{\bf c}^{-1}}} = w} \, \Delta_{2} ({\bf a} {\bf c}, {\bf b}, {\bf f}_k)
(N {\bf a} {\bf c})^{n-1} N {\bf b}^{n-1} \sigma_{({\bf a} {\bf c})^{-1}}} = \\
\nonumber
=(\xi_{w, k} \ast \beta_{k}^{{\ast \, n-1}})^{\Theta_{n} ({\bf b}, {\bf f}_k)}.
\end{eqnarray}
%$$ =
%{\sum_{a \, {\rm{mod}} \, f_k}}'
%\,\, (\xi_{w, k}
%\ast \beta_{k}^{{\ast \, n-1}})^{\sum_{c \, {\rm{mod}} \, f_k, \,
%w^{\sigma_{c^{-1}}} = w} \, \Delta_{2} (a c, b, f) (Nac)^{n-1} Nb^{n-1}
%\sigma_{(ac)^{-1}}} \,\, = $$
By the first commutative diagram of this proof and the equalities above, we obtain

$$ \partial_{F} (\Lambda (\xi_{v, k}; l^k)) =
N (\partial_{E} ( x_{w, k}^{\Theta_{1} ({\bf b}, {\bf f}_k)} \ast
\beta_{k}^{\ast \, n-1})^{N{\bf b}^{n-1}}) =
N ((\xi_{w, k} \ast \beta_{k}^{{\ast \, n-1}})^{\Theta_{n} ({\bf b}, {\bf f}_k)}) = $$
$$ = (N (\xi_{w, k} \ast \beta_{k}^{{\ast \, n-1}}))^{\Theta_{n} ({\bf b}, {\bf f})} $$
The last equality above is a result of the following well-known commutative diagram.
$$\xymatrix{
K_{2n-1} (k_{w}; \, \Z/l^k ) \quad \ar@<0.1ex>[d]^{N_{w/v}}
\ar[r]^{\sigma_{{\bf a}}^{-1}} & \quad
K_{2n-1} (k_{w^{\sigma_{{\bf a}}^{-1}}}; \, \Z/l^k)
\ar@<0.1ex>[d]^{N_{w^{\sigma_{{\bf a}}^{-1}} / v^{\sigma_{{\bf a}}^{-1}}}}\\
K_{2n-1} (k_{v}; \, \Z/l^k ) \quad
\ar[r]^{\sigma_{{\bf a}}^{-1}} & \quad
\, K_{2n-1} (k_{v^{\sigma_{{\bf a}}^{-1}}}; \, \Z/l^k)}
\label{diagram 2.8}$$

\end{proof}
Observe that for every $m > 0$ and every prime $l$,
the Bockstein exact sequence and results of Quillen
\cite{Q2}, \cite{Q3} give
natural isomorphisms

\begin{equation}
K_{m} ({\mathcal O}_{F, S})_l \cong
\varprojlim_{k} K_{m} ({\mathcal O}_{F, S} ;\, \Z/l^k),
\label{invlimringofintegers}\end{equation}
\begin{equation}
K_{m} (k_v)_l \cong
\varprojlim_{k} K_{m} (k_v ;\, \Z/l^k)\,.
\label{invlimfinitefields}\end{equation}
We define $\Lambda (\xi_{v})
\in K_{2n} ({\mathcal O}_{F, S})_l$
to be the element corresponding to
$(\Lambda (\xi_{v, k}; l^k))_{k} \in
\varprojlim_{k} K_{2n} ({\mathcal O}_{F, S} ;\, \Z/l^k)$
and define $\xi_v \in K_{2n-1} (k_v)_l$
to be the element corresponding to
$(N(\xi_{w, k} \ast \beta_{k}^{\ast \, n-1}))_k \in
\varprojlim_{k} K_{2n-1} (k_v ;\, \Z/l^k)$ via these isomorphisms, respectively.

\begin{proposition}\label{Proposition 4.5} For every $v$ such that
$l\mid q_{v}^n - 1$ and for all $k \geq k(v),$ there are homomorphisms

$$ \Lambda_{v, \, l^k} \, : \, K_{2n-1} (k_v ;\, \Z/l^k) \rightarrow
K_{2n} ({\mathcal O}_{F, S} ;\, \Z/l^k)$$
satisfying the equality
$$\Lambda_{v, \, l^k} \, (N(\xi_{w, k} \ast \beta_{k}^{\ast \, n-1}))
\,\, = \,\, \Lambda (\xi_{v, k}; l^k).$$
\end{proposition}

\begin{proof}
The definition of $\Lambda_1$ (see the beginning of section 4)
combined with the natural
isomorphism $K_1(k_w) / l^k \cong K_1(k_w; \Z/l^k\Z)$ and the natural
monomorphism $K_1({\mathcal O}_{E, S}) / l^k \to K_1({\mathcal O}_{E, S}; \Z/l^k\Z),$
coming from the corresponding Bockstein exact sequences, leads to the following homomorphism
$$\Lambda_1 : K_1(k_w; \Z/l^k\Z) \to  K_2({\mathcal O}_{E,S}; \Z/l^k\Z).$$
Multiplying on the target and on the source of this homomorphism with the $n-1$ power
of the Bott element and applying the natural isomorphism:
$$\xymatrix
{K_1(k_w; \Z/l^k\Z)\ar[r]_{\sim}^{{\ast\beta_k^{\ast n-1}}} &K_{2n-1}(k_w; \Z/l^k\Z)}$$
(see section 3)
show that there exists a unique homomorphism
$$\Lambda_{1} \ast \beta_{k}^{\ast \, n-1}\, :\,K_{2n-1} (k_w ;\, \Z/l^k) \rightarrow
K_{2n} ({\mathcal O}_{E, S} ;\, \Z/l^k),$$
sending  $ \xi_{w, k} \ast \beta_{k}^{\ast \, n-1} \rightarrow
x_{w, k}^{\Theta_{1} ({\bf b}, {\bf f}_{k})} \ast \beta_{k}^{\ast \, n-1}$. (Note that the Bott elements $\beta_k$
showing up in the left and right of the equality above live in $K_2(k_w; \, \Z/l^k\Z)$ and
$K_2(O_{E, S}; \, \Z/l^k\Z)$, respectively.)
Next, we compose the homomorphisms $\Lambda_{1} \ast \beta_{k}^{\ast \, n-1}$ defined above and
$$Tr_{E/F}\, : \, K_{2n} ({\mathcal O}_{E, S} ;\, \Z/l^k) \rightarrow
K_{2n} ({\mathcal O}_{F, S} ;\, \Z/l^k)$$
to obtain the following homomorphism:
$$ Tr_{E/F}\circ (\Lambda_{1} \ast \beta_{k}^{\ast \, n-1}): K_{2n-1} (k_w ;\, \Z/l^k) \rightarrow
K_{2n} ({\mathcal O}_{F, S} ;\, \Z/l^k)\,.$$
This homomorphism factors through the quotient of $G(k_w/k_v)$--coinvariants
$$K_{2n-1} (k_w ;\, \Z/l^k)_{G(k_{w} / k_{v})}:=
K_{2n-1} (k_w ;\, \Z/l^k) / K_{2n-1} (k_w ;\, \Z/l^k)^{Fr_v - Id}\,,$$
where $Fr_{v} \in G(k_{w} / k_{v})\subseteq G(E/F)$ is the Frobenius element
of the prime $w$ over $v.$ Since $Fr_v$ acts via $q_v^n$--powers on $K_{2n-1}(k_w)$,
$K_{2n-1} (k_w ;\, \Z/l^k)\cong K_{2n-1}(k_w)/l^k$ (see section 3) and $k\geq k(v)$, we have
$$K_{2n-1} (k_w ;\, \Z/l^k)_{G(k_{w} / k_{v})} \cong K_{2n-1} (k_w ;\, \Z/l^k) / l^{k(v)}\cong
K_{2n-1} (k_w)/l^{k(v)}.$$
The obvious commutative diagram with surjective vertical morphisms (see \S 3)
$$\xymatrix{
K_{2n-1} (k_w)/l^k \quad \ar@<0.1ex>[d]^{N_{w /v}}
\ar[r]^{\cong}  \quad &
\quad K_{2n-1} (k_w ;\, \Z/l^k)
\ar@<0.1ex>[d]^{N_{w /v}}\\
K_{2n-1} (k_v)/l^k \quad
\ar[r]^{\cong} &
\quad \, K_{2n-1} (k_v ;\, \Z/l^k)}
\label{diagram 2.51}$$
combined with the last isomorphism above, gives an isomorphism
$$\xymatrix{K_{2n-1} (k_w ;\, \Z/l^k)_{G(k_{w} / k_{v})}
\ar[r]^{\quad N_{w /v}}_{\quad\sim} &
K_{2n-1} (k_v ;\, \Z/l^k)}$$
Now, the required homomorphism
\begin{equation}
\xymatrix{\Lambda_{v, \, l^k}\, : \, K_{2n-1} (k_v ; \, \Z/l^k) \ar[r] & K_{2n} ({{\mathcal O}_{F, S}} ; \, \Z/l^k) }
\label{tag 4.6}
\end{equation}
is defined by $$\Lambda_{v, \, l^k}(x):= [Tr_{E/F}\circ (\Lambda_{1} \ast \beta_{k}^{\ast \, n-1})\circ {N_{w /v}}^{-1}(x)]^{N{\bf b}^{n-1}}\,,$$
for all $x\in K_{2n-1} (k_v ; \, \Z/l^k)$.
By definition, this map sends $N(\xi_{w, k} \ast \beta_{k}^{\ast \, n-1})$ onto the element
$\Lambda (\xi_{v, k}; l^k):=Tr_{E/F}(x_{w, k}^{\Theta_{1} ({\bf b}, {\bf f}_{k})} \ast \beta_{k}^{\ast \, n-1})^{N{\bf b}^{n-1}}$.
\end{proof}
\bigskip

\noindent It is very easy to see that the homomorphisms $\Lambda_{v, \, l^k}$ constructed above are compatible with the coefficient
reduction maps $r_{k^{\prime}/k}$, for all $k'\geq k\geq k(v)$. This permits us to construct homomorphisms
$$ \Lambda_{v}:=\varprojlim_{k} \Lambda_{v, \, l^k}:
K_{2n-1} (k_v)_l \rightarrow K_{2n} ({\mathcal O}_{F, S})_l,
$$
for all $v$ as in the last proposition. Observe that $K_{2n}({\mathcal O}_{F, S}) \subset K_{2n} (F)$, hence we can
think of the maps $\Lambda_v$ as
$$\Lambda_{v} \, : \, K_{2n-1} (k_v)_l \rightarrow K_{2n} (F)_l.$$

\begin{definition}\label{Definition 4.6}
We define the map $\Lambda$ by
$$\Lambda \, : \, \bigoplus_{v} K_{2n-1} (k_v)_l
\, \rightarrow K_{2n} (F)_l$$
$$\Lambda:=\Lambda_n \, := \, \prod_{v} \Lambda_v.$$
\end{definition}

\noindent
\section{Main results}

\begin{theorem}\label{Theorem 5.1}
The map $\Lambda:=\Lambda_n$ satisfies the following property.
$$\partial_{F} \circ \Lambda (\xi_{v}) =
\xi_{v}^{\Theta_{n} ({\bf b}, {\bf f})}
$$
\end{theorem}

\begin{proof}
Consider the following commutative diagram.
$$\xymatrix{
K_{2n} ({\mathcal O}_{F, S})/l^k \quad \ar@<0.1ex>[d]^{}
\ar[r]^{\bigoplus_{v \in S} \, \partial_{v}} \quad &
\quad\quad
\bigoplus_{v \in S}  K_{2n-1} (k_v)/l^k
\ar@<0.1ex>[d]^{}\\
K_{2n} ({\mathcal O}_{F, S} ;\, \Z/l^k) \quad
\ar[r]^{\bigoplus_{v \in S} \, \partial_{v}} &
\quad\quad\quad \bigoplus_{v \in S} \, K_{2n-1} (k_v ;\, \Z/l^k)}
\label{diagram 5.20}$$
The vertical arrows in the diagram come from the Bockstein
exact sequence. It is clear from the diagram that the inverse limit
over $k$ of the bottom horizontal arrow
gives the boundary map $\partial_{F} =
\bigoplus_{v \in S} \partial_{v}:$
$$
\partial_{F} \, :\, K_{2n} ({\mathcal O}_{F, S})_l  \rightarrow
\bigoplus_{v}  K_{2n-1} (k_v)_l.
$$
Now, the theorem follows from Propositions \ref{Proposition 4.4} and \ref{Proposition 4.5}.
\end{proof}
\medskip

In the next proposition we will construct a Stickelberger splitting map $\Gamma$
which is complementary to the map $\Lambda$ constructed above.
$$0 \stackrel{}{\longrightarrow} K_{2n} ({\mathcal O}_{F})_l
{{\stackrel{i}{\longrightarrow}} \atop {\stackrel{\Gamma}{\longleftarrow}}}
K_{2n} (F)_l
{{\stackrel{\partial_F}{\longrightarrow}} \atop {\stackrel{\Lambda}{\longleftarrow}}}
\bigoplus_{v} K_{2n-1} (k_v)_l
\stackrel{}{\longrightarrow} 0.$$

%\begin{equation}
%\xymatrix{ 0 \,\, \ar@{->}[r]&  K_{2n} ({\mathcal O}_{F})_l  \ar@{->}[r]^{i} &
%K_{2n} (F)  \ar@{->}[r]^{\partial} & \bigoplus_{v} K_{2n-1} (k_v)_l  \ar@{->}[r]& 0}
%\label{tag 4.8}
%\end{equation}

\begin{proposition}\label{Proposition 5.2} The existence of the map
$\Lambda $ satisfying the property
$\partial_{F} \circ \Lambda (\xi_{v}) =
\xi_{v}^{\Theta_{n} ({\bf b}, {\bf f})}
$
is equivalent to the existence of the map
$\Gamma \, : \, K_{2n} (F)_l \rightarrow K_{2n} ({\mathcal O}_{F})_l$
with the property
$\Gamma \circ i (\eta) =
\eta^{\Theta_{n} ({\bf b}, {\bf f})}.$ Moreover $\Gamma \circ \Lambda = 0.$
\end{proposition}
\begin{proof}
Assume that we have the map $\Lambda.$ For any $\eta \in K_{2n} (F)_l$ define
$$\Gamma (\eta) \, := \,
\Lambda \circ \partial_{F} (\eta^{-1}) \, \eta^{\Theta_{n} ({\bf b}, {\bf f})}$$
Observe that, in principle, we have  $\Gamma (\eta) \in K_{2n} (F)_l.$ However, by
Theorem \ref{Theorem 5.1},
$$\partial_{F} (\Gamma (\eta)) = \partial_{F} ( \Lambda \circ \partial_{F} (\eta^{-1})
\, \eta^{\Theta_{n} ({\bf b}, {\bf f})}) = \partial_{F} (\eta)^{ - \Theta_{n} ({\bf b}, {\bf f})}
\partial_{F} (\eta)^{\Theta_{n} ({\bf b}, {\bf f})} = 1.$$
Hence $\Gamma (\eta) \in K_{2n} ({\mathcal O}_{F})_l.$ Moreover, for $\eta \in
K_{2n} ({\mathcal O}_{F})_l$, we have
$$\Gamma \circ i (\eta) = \Lambda \circ \partial_{F} (i(\eta))^{-1} \, i(\eta)^{\Theta_{n} ({\bf b}, {\bf f})} = \eta^{\Theta_{n} ({\bf b}, {\bf f})}.$$

Now, assume that we have the map $\Gamma.$ For any $(\xi_v) \in \bigoplus_{v} K_{2n-1} (k_v)_l$ define
$$\Lambda ((\xi_v)) \, := \,
\Gamma (\eta^{-1}) \, \eta^{\Theta_{n} ({\bf b}, {\bf f})},$$
where $\eta \in K_{2n} (F)_l$, such that $\partial_{F} (\eta) = (\xi_v).$
Observe that the definition does not depend on
the choice of $\eta$. Indeed, for any other $\eta^{\prime}$ such that
$\partial_{F} (\eta^{\prime}) = (\xi_v)$ we have $\eta^{\prime} \eta^{-1} \in
K_{2n} ({\mathcal O}_{F})_l.$ So $\Gamma ((\eta^{\prime} \eta^{-1} )) \, (\eta^{\prime \, -1}\eta)^{\Theta_{n} ({\bf b}, {\bf f})} = 1$ by the property of $\Gamma$
since $\eta^{\prime} \eta^{-1} = i(\eta^{\prime} \eta^{-1}).$ It is clear that
$\partial_{F} \circ \Lambda (\xi_{v}) =
\xi_{v}^{\Theta_{n} ({\bf b}, {\bf f})}.$
Moreover,  by Theorem \ref{Theorem 5.1} we have the following equalities
$$\Gamma \circ \Lambda ((\xi_v))\, = \,
\Lambda ( \partial_{F} (\Lambda ((\xi_v))^{-1})) \, \Lambda ((\xi_v))^{\Theta_{n} ({\bf b}, {\bf f})} = (\Lambda ((\xi_v)))^{- \Theta_{n} ({\bf b}, {\bf f})} \, \Lambda ((\xi_v))^{\Theta_{n} ({\bf b}, {\bf f})} = 1$$
\end{proof}

\begin{remark}\label{Remark 5.3}
Observe that this map $\Lambda$ is defined in the same way for both cases
$l \not \,\mid n$ and $l \mid n.$ If restricted to the particular case
$K=\Bbb Q$, our construction improves upon that of \cite{Ba1}. In loc. cit., in the case $l\mid n$
the map $\Lambda$ was constructed only up to a factor of $l^{v_l(n)}$.
\end{remark}

\begin{theorem}\label{Theorem 5.4}
Assume that the Stickelberger elements $\Theta_{1} ({\bf b}, {\bf f}_k)$ annihilate
the groups $K_{2} ({\mathcal O}_{F_k})_l$ for all $k \geq 1.$
Then the Stickelberger element $\Theta_{n} ({\bf b}, {\bf f})$ annihilates the group
$div \, K_{2n} (F)_l$ for all $n \geq 1.$
\end{theorem}
\begin{proof}
The proof is very similar to the proof of [Ba1, Cor. 1, p. 340].
Let $d \in div \, K_{2n} (F)_l.$ Take $m \in \N$ such that $d = x^{l^m}$ for some
$x \in K_{2n} (F)_l$ and $l^m$ annihilates $K_{2n} ({\mathcal O}_{F})_l.$
Then $\Lambda(\partial_{F} (x^{-1}))\, x^{\Theta_{n} ({\bf b}, {\bf f})} \in
K_{2n} ({\mathcal O}_{F})_l$ because
$$\partial_{F} ( \Lambda(\partial_{F} (x^{-1}))\, x^{\Theta_{n} ({\bf b}, {\bf f})} ) =
\partial_{F} (x)^{- \Theta_{n} ({\bf b}, {\bf f})}
\partial_{F} (x)^{\Theta_{n} ({\bf b}, {\bf f})} = 1$$
by Theorem \ref{Theorem 5.1} and Galois equivariance of $\partial_{F}.$
Hence $$(\Lambda(\partial_{F} (x^{-1}))\, x^{\Theta_{n} ({\bf b}, {\bf f})})^{l^m} =
\Lambda(\partial_{F} (d^{-1}))\, d^{\Theta_{n} ({\bf b}, {\bf f})} =
d^{\Theta_{n} ({\bf b}, {\bf f})} = 1$$
\end{proof}

\noindent
\begin{remark}\label{second proof of annihilation of div}
Observe that we can restrict the map $\Lambda$ to the $l^k$--torsion part, for
any $k \geq 1.$ For any $k \gg 0$, there is an exact sequence
$$
0 \stackrel{}{\longrightarrow} K_{2n} ({\mathcal O}_{F}) [l^k]
\stackrel{}{\longrightarrow} K_{2n} (F) [l^k]
{{\stackrel{\partial_F}{\longrightarrow}} \atop
{{\stackrel{\Lambda}{\longleftarrow}}}} \bigoplus_{v} K_{2n - 1} (k_v) [l^k]
\stackrel{}{\longrightarrow} div
(K_{2n} (F)_l) \stackrel{}{\longrightarrow} 0
$$
By Theorem \ref{Theorem 5.1}, we know that $\partial_F \circ \Lambda$
is the multiplication by $\Theta_{n} ({\bf b}, {\bf f}).$
As pointed out in the Introduction, this implies the annihilation of
$div\,( K_{2n} (F)_l)$ and consequently gives a second proof for Theorem
\ref{Theorem 5.4}
\end{remark}
\bigskip

\noindent
Let us define $F_0 := F$ and:
$$
\Theta_{n} ({\bf b}, {\bf f}_{0}) \,\, = \,\,
\left\{
\begin{array}{lll}
\bigl( \, \prod_{{{\bf l} \, \not \, | \, {\bf f}} \atop
{{\bf l} \, | \, l}} \, (1 - ({\bf l}, \, F)^{-1} N{\bf l}^{n}) \, \bigr) \,\,
\Theta_{n} ({\bf b}, {\bf f}) &\rm{if}&l \not| {\bf f}\\
\Theta_{n} ({\bf b}, {\bf f}) &\rm{if}& l | {\bf f}\\
\end{array}\right.
$$
Hence by the formula (\ref{2.53}) we get
\begin{equation}
Res_{F_{k+1} / F_{k}} \,\, \Theta_{n} ({\bf b}, {\bf f}_{k+1})
=  \,\, \Theta_{n} ({\bf b}, {\bf f}_k)
\label{RestofStickinIwasawaTower}\end{equation}
Hence by formula (\ref{RestofStickinIwasawaTower}) we can define the element
\begin{equation}
\Theta_{n} ({\bf b}, {\bf f}_{\infty}) :=
\varprojlim_{k} \Theta_{n} ({\bf b}, {\bf f}_{k}) \in \varprojlim_{k}\, \Z_l [G(F_k / F)].
\label{InfiniteStickelberger}\end{equation}

\begin{corollary}\label{Corollary 5.5}
Assume that the Stickelberger elements
$\Theta_{1} ({\bf b}, {\bf f}_k)$ annihilate
the groups $K_{2} ({\mathcal O}_{F_k})_l$ for all $k \geq 1.$
Then the Stickelberger element $\Theta_{n} ({\bf b}, {\bf f}_k)$ annihilates the group
$div \, K_{2n} (F_k)_l$ for every $k \geq 0$ and every $n \geq 1.$
In particular $\Theta_{n} ({\bf b}, {\bf f}_{\infty})$ annihilates the group
$\varinjlim_{k} div \, K_{2n} (F_{k})_l$ for every $n \geq 1.$
\end{corollary}
\begin{proof}
Follows immediately from Theorem \ref{Theorem 5.4}.
\end{proof}

\begin{theorem}\label{Theorem 5.6}
Let $F/\Q$ be an abelian extensions of conductor $f.$ Let an integer $b$ be prime
to $w_{n+1}(\Q(\mu_{lf}))|K_{2} ({\mathcal O}_{F})_l|.$ Then $\Theta_{n} (b, f)$
annihilates the group $div \, K_{2n} (F)_l$ for all $n \geq 1.$
\end{theorem}
\begin{proof} Coates and Sinnott [CS] proved that $\Theta_{1} (b, f_k)$ annihilates
$K_{2} ({\mathcal O}_{F_k})$ for all $k \geq 1.$ Hence the theorem follows by
Theorem \ref{Theorem 5.4}.
\end{proof}

\begin{remark}\label{Remark 5.7}
Observe that Theorem \ref{Theorem 5.6} strengthens [Ba1, Cor. 1, p. 340]
in the case $l \, | \,n.$
\end{remark}

A much more general consequence of Theorem 5.4 above is the following.

\begin{theorem}\label{Theorem 5.8}
Let $F/K$ be an abelian CM extension of an arbitrary totally real number field $K$
and let $l$ be an odd prime. If the Iwasawa $\mu$--invariant $\mu_{F,l}$ associated
to $F$ and $l$ vanishes, then $\Theta_{n} (\bf b, \bf f)$
annihilates the group $div (K_{2n} (F)_l)$ for all $n \geq 1.$
\end{theorem}

\begin{proof} In [GP], it is shown that if $\mu_{F, l}=0$, then  $\Theta_{n} (\bf b, \bf f)$
annihilates $K_{2n}^{et}(O_F[1/l])$, for all odd $n$. From the definition of Iwasawa's $\mu$--invariant
one concludes right away that if $\mu_{F, l}=0$, then $\mu_{F_k, l}=0$, for all $k$. Consequently,
$\Theta_{1} (\bf b, \bf f_k)$ annihilates $K_{2}^{et}(O_{F_k}[1/l])$, for all $k$. Now, one applies Tate's Theorem
\ref{Tate's Theorem} to conclude that  $\Theta_{1} (\bf b, \bf f)$ annihilates $K_{2}(O_{F_k})_l$, for all $k$.
Theorem \ref{Theorem 5.4} implies the desired result.
\end{proof}

\begin{remark} Note that Theorem 5.6 above is indeed a particular case of Theorem 5.8, as $\mu_{F, l}=0$ for all abelian extensions
$F/\Bbb Q$ and all primes $l$ (according to a classical theorem of Ferrero-Washington and Sinnott.) It is a classical conjecture of Iwasawa
that $\mu_{F, l}=0$ for all number fields $F$ and all primes $l$.
\end{remark}

\noindent
\section{Construction of the map $\Lambda^{et}$}

Since Quillen K-theory and \' etale K-theory of rings of integers
and number fields enjoy many similar properties, we can construct
the Stickelberger splitting map $\Lambda^{et}$ in the setting of
\' etale K-theory as well. This section consists of a brief
description of the key steps of the construction of
$\Lambda^{et}.$ If $R$ is either a number field $L$ or its ring of
$l$--integers ${\mathcal O}_{L, S}[1/l]$,  Tate proved in
\cite{Ta2} that there is a natural isomorphism:
$$K_{2} (R)_l \,
{\stackrel{\cong}{\longrightarrow}} \, K_{2}^{et} (R).$$ Dwyer and
Friedlander \cite{DF} proved that there are natural isomorphisms
$$K_{2} (R; \Z/l^k ) \,
{\stackrel{\cong}{\longrightarrow}} \, K_{2}^{et} (R; \Z/l^k),$$
for all $k\geq 1$. As explained in \cite{Ba2}, for any number
field $L$ and any finite set $S \subset \text{Spec} ({\mathcal
O}_{L})$ we have the following commutative diagrams with exact
rows and (surjective) Dwyer-Friedlander maps as vertical arrows.

$$\xymatrix{
0  \ar[r]^{} & K_{2n} ({\mathcal O}_{L})_l
\ar@{>>}[d]^{} \ar[r]^{}  & K_{2n} ({\mathcal O}_{L, S})_l
\ar@{>>}[d]^{} \ar[r]^{}  & \bigoplus_{v \in S} \, \,
K_{2n-1} (k_v)_l \ar@<0.1ex>[d]^{\cong} \ar[r]^{} & 0 \\
0  \ar[r]^{} & K_{2n}^{et} ({\mathcal O}_{L}[1/l])  \ar[r]^{}  &
K_{2n}^{et} ({\mathcal O}_{L, S}[1/l]) \ar[r]^{}  & \bigoplus_{v
\in S} \, \, K_{2n-1}^{et} (k_v) \ar[r]^{} & 0}
\label{CommDiagrQuilEtale1}$$ For $n = 1$, the left and the middle
vertical arrows in the above diagram are also isomorphisms,
according to Tate's theorem.
\medskip

We assumed throughout this paper that $\Theta_{1} ({\bf b}, {\bf
f}_k)$ annihilates $K_{2} ({\mathcal O}_{F_{l^k}})$ for all $k
\geq 0.$ Hence $\Theta_{1} ({\bf b}, {\bf f}_k)$ annihilates
$K_{2}^{et} ({\mathcal O}_{F_{l^k}}[1/l])$ for all $k \geq 0.$
Recall the construction of $\Lambda_1$ just before Lemma
\ref{Lemma 4.1}. In the diagram above, let $y_{w, k}$ and
$\zeta_{w, k}$ denote the images of  $x_{w, k}$ and $\xi_{w, k}$
via the middle vertical and right vertical arrows, respectively.
Then, we define
$$\Lambda_{1}^{et} (\zeta_{w, k}) \, := \, y_{w, k}^{\Theta_{1} ({\bf b}, {\bf f}_k)}.$$
Clearly, the following diagram is commutative.
$$\xymatrix{
K_{2} ({\mathcal O}_{F_{l^k}, S})_l
\ar@<0.1ex>[d]^{\cong}\,   &  \ar[l]_{\Lambda_1} \bigoplus_{v \in S} \bigoplus_{w | v}
\, K_{1} (k_w)_l \ar@<0.1ex>[d]^{\cong} \\
K_{2}^{et} ({\mathcal O}_{F_{l^k}, S}[1/l]) &
\ar[l]_{\Lambda_{1}^{et}}  \bigoplus_{v \in S} \bigoplus_{w | v}
\, K_{1}^{et} (k_w)_l} \label{CommDiagrQuilEtale3}$$
We define
elements $\Lambda^{et} (\xi_{v, k}; l^k) \in K_{2n}^{et}
({\mathcal O}_{F, S};\, \Z/l^k))$ as follows:
$$\Lambda^{et} (\zeta_{v, k}; l^k) := \Lambda_{n}^{et} (\zeta_{v, k}; l^k) :=
Tr_{E/F} ( y_{w, k}^{\Theta_{1} ({\bf b}, {\bf f}_k) } \ast
\beta_{k}^{\ast \, n-1})^{N {\bf b}^{n-1}}.$$ Obviously,
$\Lambda^{et} (\zeta_{v, k}; l^k)$ is the image of $\Lambda
(\xi_{v, k}; l^k)$ via the Dwyer-Friedlander map. Now analogs of
Lemmas 4.1, 4.2 and 4.3 and Propositions 4.4 and 4.5 hold for the
\'etale case with  $\Lambda$ replaced by $\Lambda^{et},$ $\xi_{v,
k}$ replaced by $\zeta_{v, k},$ and $x_{w, k}$ replaced by $y_{w,
k}$ etc. Observe that the result of Gillet \cite{Gi} for K-theory
discussed in \S3 is replaced by the compatibility of the
Dwyer-Friedlander spectral sequence with the product structure
(\cite{DF} Proposition 5.4) and by Soul{\' e}'s observation (see
\cite{So1} p. 275) that the localization sequence in \' etale
cohomology (see \cite{So1} p. 268) is compatible with the product
by \' etale cohomology of ${\mathcal O}_{F, S}.$ Eventually, these
observations allow us to construct the map
$$\Lambda^{et} \, : \, \bigoplus_{v} K_{2n-1}^{et} (k_v)
\, \rightarrow K_{2n}^{et} (F)_l$$ which is the \'etale analogue
of our map $\Lambda $ from \S\S4-5. Naturally, by construction,
the following diagram commutates.
$$\xymatrix{
K_{2n}(F)_l
\ar@<0.1ex>[d]^{}\,   &  \ar[l]_{\Lambda} \bigoplus_{v}
\, K_{2n-1} (k_v)_l \ar@<0.1ex>[d]^{\cong} \\
K_{2n}^{et} (F)_l &  \ar[l]_{\Lambda^{et}}  \bigoplus_{v} \,
K_{2n-1}^{et} (k_v)_l} \label{CommDiagrQuilEtale4}$$ Moreover, the
discussion above shows that we have the following \' etale
analogue of Theorem \ref{Theorem 5.1}.
\begin{theorem}\label{Theorem 6.1}
The map $\Lambda^{et}$ satisfies the following property.
$$\partial_{F}^{et} \circ \Lambda^{et} (\zeta_{v}) =
\zeta_{v}^{\Theta_{n} ({\bf b}, {\bf f})}
$$
\end{theorem}
\noindent The following is the \'etale analogue of Theorem 5.4.
\begin{theorem}\label{Theorem 6.2}
Assume that the Stickelberger elements $\Theta_{1} ({\bf b}, {\bf f}_k)$ annihilate
the groups $K_{2}^{et} ({\mathcal O}_{F_k})_l$ for all $k \geq 1.$
Then the Stickelberger element $\Theta_{n} ({\bf b}, {\bf f})$ annihilates the group
$div \, K_{2n}^{et} (F)_{l}$ for all $n \geq 1
.$
\end{theorem}
\begin{proof} The proof is identical to that of Theorem
\ref{Theorem 5.4} with $\Lambda$ replaced by $\Lambda^{et}.$ One
also observes that this theorem follows more directly from
Theorem \ref{Theorem 5.4} since by \cite{Ba2} Theorem 3(i) we know that $div \, K_{2n}
(F)_{l}$ is isomorphic to $div \, K_{2n}^{et} (F)_{l}$ via the
Dwyer-Friedlander map $K_{2n} (F)_l \rightarrow K_{2n}^{et}
(F)_l.$
\end{proof}

\begin{remark}\label{Remark 6.3}
Based on Theorems \ref{Theorem 6.1} and \ref{Theorem 6.2}, we can easily establish
\'etale versions of Remark \ref{second proof of annihilation of div},
Corollary \ref{Corollary 5.5} and Theorems \ref{Theorem 5.6} and  \ref{Theorem 5.8}.
\end{remark}

\noindent
\section{Construction of $\Lambda$ and $\Lambda^{et}$ revisited}

In this section we will generalize our approach used in \S4 and construct the map
$\Lambda^{\prime} := \Lambda_{n}^{\prime} $ for $K_{n}$, under the
assumption that for some fixed $m > 0$ the Stickelberger element
$\Theta_{m} ({\bf b}, {\bf f}_k)$ annihilates
$K_{2m} ({\mathcal O}_{F_{l^k}})$  for all $k \geq 0.$
Since the construction is similar to the those in
\S4, we will only sketch the proofs of these results.
\medskip

\noindent
Let $L$ be a number field, such that $\mu_{l^k} \subset {\mathcal O}_{L, S}.$
Let $i \in \N$ and let $m \in \Z.$
Then, for $R = L$ or $R = {\mathcal O}_{L, S}$ there is a natural group isomorphism
\cite{DF} Theorem 5.6:
\begin{equation}
K_{i}^{et} (R; \Z/l^k)
\, {\stackrel{\cong}{\longrightarrow}} \,
K_{i + 2 m}^{et} (R; \Z/l^k)
\label{etalebottmult}\end{equation}
which sends $\eta$ to
$\eta \ast \beta_{k}^{\ast \,  m}$ for any
$\eta \in K_{i}^{et} (R; \Z/l^k).$
If $m \geq 0$ this isomorphism is just the multiplication by
$\beta_{k}^{\ast \, m}.$
If $m < 0$ and $i + 2m > 0$, then the isomorphism (\ref{etalebottmult}) is
the inverse to the multiplication by $\beta_{k}^{\ast \, - m}$ isomorphism:
\begin{equation}
\ast \,\, \beta_{k}^{\ast \, - m}\,  : \,
K_{i + 2m}^{et}  (R; \Z/l^k) \,\, {\stackrel{\cong}{\longrightarrow}} \,\,
K_{i}^{et} (R; \Z/l^k).
\label{etalebottmult1}\end{equation}
\medskip

\noindent
Now, let us consider Quillen K-theory.
If $m \geq 0$, there is a natural homomorphism
\begin{equation}
\ast \, \, \beta^{\ast \, m} \, : \: K_{i}  (R; \Z/l^k)
\rightarrow K_{i + 2 m} (R; \Z/l^k)
\label{ktheorybottmult0}\end{equation}
which is just a multiplication by  $\beta_{k}^{\ast \, m}.$
The homomorphism (\ref{ktheorybottmult0}) is compatible with the isomorphism
(\ref{etalebottmult}) via the Dwyer-Friedlander map.
If $m < 0$ and $i + 2m > 0$,  then take the homomorphism
\begin{equation}
t(m) \, : \: K_{i}  (R; \Z/l^k)
\rightarrow K_{i + 2 m} (R; \Z/l^k)
\label{ktheorybottmult00}\end{equation}
to be the the composition
of the left vertical, bottom horizontal and right vertical arrows of the following diagram.
$$\xymatrix{
K_{i}  (R; \Z/l^k)
\ar@<0.1ex>[d]^{}\,  \quad \ar[r]^{t (m)}  &  \quad K_{i + 2 m}  (R; \Z/l^k)   \\
K_{i }^{et}  (R; \Z/l^k) \quad
\ar[r]^{(\ast \, \beta_{k}^{\ast \, - m})^{-1}}
&  \quad K_{i + 2 m}^{et}  (R; \Z/l^k)  \ar@<0.1ex>[u]^{} }
\label{ktheorybottmult}$$
The left vertical arrow is the Dwyer-Friedlander map. The right vertical arrow is
the Dwyer-Friedlander splitting \cite{DF}, Proposition 8.4.
The Dwyer-Friedlander splitting map is obtained as the multiplication of the
inverse to the isomorphism
$K_{i^{\prime}}  (R; \Z/l^k) {\stackrel{\cong}{\longrightarrow}}
K_{i^{\prime}}^{et}  (R; \Z/l^k),$
for $i^{\prime} = 1 $ or $i^{\prime} = 2,$ by a nonnegative power of the Bott element
$\beta_{k}^{\ast \, m^{\prime}},$ with $m^{\prime} \geq 0$
(see the proof of \cite{DF}, Proposition 8.4.)

\begin{remark} \label{Remark 7.1 on coeff change oF DF splitting}
It is clear that the Dwyer-Friedlander splitting
from \cite{DF}, Proposition 8.4 is compatible with the maps
$\Z/l^{j} \rightarrow \Z/l^{j-1}$ at the level of coefficients, for all $1 \leq j \leq k.$
Consequently, the map $t(m)$ is naturally compatible with these maps.
In addition, $t(m)$ is naturally compatible with the ring imbedding
$R \rightarrow R^{\prime}$, where $R^{\prime} = L^{\prime}$ or
$R^{\prime} = {\mathcal O}_{L^{\prime}, S}$ for a number field extension
$L^{\prime}/L.$ Let
$$t^{et} (m) := (\ast \, \beta_{k}^{\ast \, - m})^{-1}.$$
It is clear from the above diagram that $t(m)$ and $t^{et} (m)$ are
naturally compatible with the Dwyer-Friedlander maps.
\end{remark}

\begin{lemma}\label{Lemma 7.2}
Let $L = F (\mu_{l^k})$ and let $i > 0$ and $m < 0$, such that $i + 2 m > 0.$
Then, for $R = L$ or $R = {\mathcal O}_{L, S}$, the natural group homomorphisms
$t^{et} (m)$  and $t(m)$ have the following properties:
\begin{equation}
t^{et} (m) (\alpha)^{\sigma_{{\bf a}}} =
t^{et} (m) (\alpha^{N  {\bf a}^m \sigma_{{\bf a}}})
\label{ktheorybottmult1}\end{equation}
\begin{equation}
t(m) (\alpha)^{\sigma_{{\bf a}}} =
t(m) (\alpha^{N  {\bf a}^m \sigma_{{\bf a}}})
\label{ktheorybottmult1second}\end{equation}
for $\alpha \in K_{i}^{et} (R; \Z/l^k)$ (resp.
$\alpha \in K_{i} (R; \Z/l^k)$).
\end{lemma}

\begin{lemma}\label{Lemma 7.3}
If $i \in \{1, 2\},$ $\alpha \in K_{i} (R; \Z/l^k)$ and $n + m > 0$
then
\begin{equation}
t^{et} (m) (\alpha \ast \beta_{k}^{\ast \, n }) =
\alpha \ast \beta_{k}^{\ast \, n + m }.
\label{etalenegativebottmult2}\end{equation}
\begin{equation}
t (m) (\alpha \ast \beta_{k}^{\ast \, n }) =
\alpha \ast \beta_{k}^{\ast \, n + m }.
\label{ktheorynegativebottmult2}\end{equation}
\end{lemma}
\begin{proof} The properties in Lemmas \ref{Lemma 7.2} and \ref{Lemma 7.3}
follow directly from the definition of the
maps $t^{et} (m)$ and $t(m).$
\end{proof}

If $v$ is a prime of ${\mathcal O}_{L, S},$
$m < 0$ and $i + 2m > 0$, then we construct the morphism
\begin{equation}
t_v (m) \, : \: K_{i}  (k_v ; \Z/l^k)
\rightarrow K_{i + 2 m} (k_v ; \Z/l^k)
\label{ktheorybottmult00second}\end{equation}
in the same way as we have done for ${\mathcal O}_{L, S}$ or $L.$
Namely, $t_v (m)$ is the composition
of the left vertical, bottom horizontal and right vertical arrows in the following diagram.
$$\xymatrix{
K_{i}  (k_v ; \Z/l^k)
\ar@<0.1ex>[d]^{\cong}\,  \quad \ar[r]^{t_v (m)}  &  \quad K_{i + 2 m}  (k_v ; \Z/l^k)   \\
K_{i}^{et}  (k_v; \Z/l^k) \quad
\ar[r]^{(\ast \, \beta_{k}^{\ast \, - m})^{-1}}
&  \quad K_{i + 2 m}^{et}  (k_v; \Z/l^k)  \ar@<0.1ex>[u]^{\cong} }
\label{ktheorybottmult3}$$
The right vertical arrow is
the inverse of the Dwyer-Friedlander map which, in the case of a finite field,
is clearly seen to be equal to the Dwyer-Friedlander splitting map.

Similarly to $t^{et} (m)$ we can construct $t_{v}^{et} (m).$
We observe that the maps $t (m)$ and $t_v (m)$ are compatible with the reduction
maps and the boundary maps. In other words, we have the following commutative diagrams.

$$\xymatrix{
K_{i}  ({\mathcal O}_{L, S} ; \, \Z/l^k)
\ar@<0.1ex>[d]^{t (m)}\,  \quad \ar[r]^{r_v}  &  \quad K_{i }  (k_v; \, \Z/l^k)
 \ar@<0.1ex>[d]^{t_v (m)}  \\
K_{i + 2 m}  ({\mathcal O}_{L, S} ; \, \Z/l^k) \quad
\ar[r]^{r_v}
&  \quad K_{i + 2 m}  (k_v; \, \Z/l^k)}
\label{ktheorybottmult4}$$
\medskip

\noindent
$$\xymatrix{
 K_{i} ({\mathcal O}_{L, S}, \, \Z/l^k)
\ar@{>}[d]^{t (m)} \ar[r]^{\partial}  & \bigoplus_{v \in S} \, \,
K_{i-1} (k_v; \, \Z/l^k) \ar@<0.1ex>[d]^{t_v (m)} \\
K_{i + 2m} ({\mathcal O}_{L, S}[1/l]; \, \Z/l^k) \ar[r]^{\partial}  & \bigoplus_{v
\in S} \, \, K_{i - 1 + 2m} (k_v; \, \Z/l^k) }
\label{CommDiagrQuilEtale1second}$$
\medskip
\noindent
Let us point out that there are similar diagrams to the two diagrams above
for \'etale K-theory and the maps $t^{et} (m)$ and $t_{v}^{et} (m).$

\medskip
As observed in the discussion above, the map $t(m)$ for $m < 0$ has
the same properties as the multiplication by $\beta^{\ast \, m}$
for $m \geq 0.$ So, for $m < 0$ we define the symbols $\alpha \ast
\beta^{\ast \, m} : = t (m) (\alpha)$
(resp. $\alpha_v \ast
\beta^{\ast \, m} : = t_v (m) (\alpha_v)$, for
$\alpha \in K_{i} ({\mathcal O}_{L}; \, \Z/l^k)$
(resp. $\alpha_v \in K_{i} (k_v; \Z/l^k)).$
For $m \geq 0$, the symbol $\alpha \ast \beta^{\ast \, m}$
(resp. $\alpha_v \ast \beta^{\ast \, m}$) denotes the usual thing.
\medskip

\medskip

Let $m  > 0$ be a natural number. Throughout the rest of this section
we assume that $\Theta_{m} ({\bf b}, {\bf f}_k)$ annihilates
$K_{2m} ({\mathcal O}_{F_{l^k}})$ for all $k \geq 0.$
As \S4, we let $w$ denote a prime of ${\mathcal O}_{F_{l^k}}$ over a prime $v$
of ${\mathcal O}_{F}$, such that $v \not\,\mid l.$ Put $E := F_{l^k}.$
For any finite set $S$ of primes in $O_F$ and any $k\geq 0$,
there is an exact sequence [Q].

$$
0 \stackrel{}{\longrightarrow} K_{2 m} ({\mathcal O}_{F_{l^k}})
\stackrel{}{\longrightarrow} K_{2 m} ({\mathcal O}_{F_{l^k}, \, S})
\stackrel{\partial}{\longrightarrow} \bigoplus_{v \in S} \bigoplus_{w | v}
K_{2 m -1} (k_w) \stackrel{}{\longrightarrow} 0
$$

Let $\xi_{w, k} \in K_{2 m - 1} (k_w)_l$ be a generator of the $l$-torsion part of
$K_{2 m -1} (k_w)$.
Pick an element $x_{w, k} \in K_{2 m} ({\mathcal O}_{F_{l^k}, S})_l$ such that
${\partial} (x_{w, k}) = \xi_{w, k}.$ Obviously,
$x_{w, k}^{\Theta_{m} ({\bf b}, {\bf f}_k)}$ does not depend on the choice of $x_{w, k}$
since $\Theta_{m} ({\bf b}, {\bf f}_k)$ annihilates
$K_{2 m} ({\mathcal O}_{F_{l^k}}).$ If $\text{ord}(\xi_{w, k}) = l^a$, then
$x_{w, k}^{l^a} \in K_{2 m} ({\mathcal O}_{F_{l^k}}).$ Hence,
$(x_{w, k}^{\Theta_{m} ({\bf b}, {\bf f}_k)})^{l^a} =
(x_{w, k}^{l^a})^{\Theta_{m} ({\bf b}, {\bf f}_k)} = 0.$
Consequently,  there is a well defined map:

$$\Lambda_{m}^{\prime}\, : \, \bigoplus_{v \in S} \bigoplus_{w | v} K_{2 m - 1} (k_w)_l
\stackrel{}{\longrightarrow} K_{2 m} ({\mathcal O}_{F_{l^k}, \, S})_l ,$$
$$\Lambda_{m}^{\prime} (\xi_{w, k}) \, := \, x_{w, k}^{\Theta_{m} ({\bf b}, {\bf f}_k)}.$$

\begin{lemma}\label{Lemma 7.4}
The map $\Lambda_{m}^{\prime}$ satisfies
the following property
$$\partial \Lambda_{m}^{\prime} (\xi_{w, k}) \, :=
\, \xi_{w, k}^{\Theta_{m} ({\bf b}, {\bf f}_k)}.$$
\end{lemma}
\begin{proof}
The lemma follows immediately by compatibility of $\partial$ with
$G (E/F)$ action.
\end{proof}

Let $v$ be a prime in ${\mathcal O}_{F}$ sitting above $p \not= l$
in $\Z$. Let $S := S_v$ be the finite set primes of ${\mathcal
O}_{F}$ consisting of all the primes over $p.$ Let us fix an
$n\in\N$. Let $k (v)$ be the natural number for which $l^{k(v)}\,
|| \, q_{v}^{n} - 1.$  For $k \geq k (v)$, let us define elements:

$$\Lambda_{n}^{\prime} (\xi_{v, k}; l^k) :=
Tr_{E/F} ( x_{w, k}^{\Theta_{m} ({\bf b}, {\bf f}_k) } \ast
\beta_{k}^{\ast \, n-m})^{N {\bf b}^{n-m}} \in K_{2n} ({\mathcal
O}_{F, S};\, \Z/l^k)).$$ As before, we will
write $\Lambda^{\prime} (\xi_{v, k}; l^k)$
instead of $\Lambda_{n}^{\prime} (\xi_{v, k}; l^k).$

Let us fix a prime sitting above $v$ in each of the fields $F(\mu_{l^k})$,
such that if $k\leq k'$ and $w$ and $w'$ are the fixed primes in $E =F(\mu_{l^k})$ and
$E':=F(\mu_{l^{k'}})$, respectively, then $w'$ sits above $w$. By the surjectivity of
the transfer maps for $K$-theory of finite fields (see the end of
\S3), we can associate to each $k$ and the chosen prime $w$ in
$E=F(\mu_{l^k})$ a generator $\xi_{w, k}$ of $K_{2m-1}(k_w)_l$, such that
$$N_{w^{\prime}/w}(\xi_{w^{\prime},
k^{\prime}})= \xi_{w, k},$$
for all $k\leq k'$, where $w$ and $w'$ are the fixed primes in $E =F(\mu_{l^k})$ and
$E' =F(\mu_{l^{k'}})$, respectively.

\begin{lemma}\label{Lemma 7.5} With notations as above, for every
$k \leq k^{\prime}$ we have
$$r_{k^{\prime}/k} ( N_{w^{\prime}/v}
(\xi_{w^{\prime}, k^{\prime}} \ast \beta_{k^{\prime}}^{\ast \, n-m})) =
N_{w/v} (\xi_{w, k} \ast \beta_{k}^{\ast \, n-m})$$
\end{lemma}

\begin{proof} For $n - m \geq 0$ the proof is the same as the proof of
Lemma \ref{Lemma 4.2}. Assume that $n - m < 0.$
Since the Dwyer-Friedlander maps commute with $N_{w/v}$ and
$N_{w^{\prime}/v}$, the proof is similar to that of
Lemma \ref{Lemma 4.2} with only slight modifications. Namely, we use
the projection formula for the negative twist in \'etale cohomology, since
for any finite field $\F_q$ with $l \not\,\, \mid   q,$ we have natural isomorphisms
coming from the Dwyer-Friedlander
spectral sequence (cf. the end of \S2)
\begin{equation}
K_{2j-1}^{et} (\F_q) \cong H^1 (\F_q; \, \Z_{l} (j))
\label{finitefieldDF1}\end{equation}
\begin{equation}
K_{2j - 1}^{et} (\F_q; \, \Z/l^k) \cong H^1 (\F_q; \, \Z/l^k (j)).
\label{finitefieldDF2}\end{equation}
\end{proof}

\begin{lemma}\label{Lemma 7.6}
For all $k(v) \leq k \leq k^{\prime}$, we have
$$r_{k^{\prime}/k} (\Lambda^{\prime} (\xi_{v, k^{\prime}};
l^{k^{\prime}})) =
\Lambda^{\prime} (\xi_{v, k}; l^k)$$
\end{lemma}

\begin{proof}
As in the proof of Lemma \ref{Lemma 4.3}, we observe that
$Tr_{E^{\prime}/E} ({x}_{w^{\prime}, k^{\prime}})^{\Theta_{m} ({\bf b}, {\bf f}_k)} =
{x}_{w, k}^{\Theta_{m} ({\bf b}, {\bf f}_k)}.$
For $n-m \geq 0$, the proof is the same as that of Lemma \ref{Lemma 4.3}.
Assume that $n - m < 0.$ We observe that $Tr_{E^{\prime}/E}$ commutes with the Dwyer-Friedlander map.
Hence $Tr_{E^{\prime}/E}$ also commutes with the splitting of the Dwyer-Friedlander map
since the splitting is a monomorphism. By the Dwyer-Friedlander spectral sequence for any
number field $L$ and any finite set $S$ of prime ideals of ${\mathcal O}_{L}$
containing all primes over $l,$
we have the following isomorphism
\begin{equation}
K_{2j}^{et} ({\mathcal O}_{L, S}) \cong H^2 ({\mathcal
O}_{L, S}; \, \Z_{l} (j + 1))
\label{ringofintegersDF1}\end{equation}
and the following exact sequence
\begin{equation}
0 \rightarrow H^2 ({\mathcal O}_{L, S}; \, \Z/l^k(j + 1))
\rightarrow K_{2j}^{et} ({\mathcal O}_{L, S};\, \Z/l^k) \rightarrow
H^0 ({\mathcal O}_{L, S}; \, \Z/l^k (j)) \rightarrow 0.
\label{finitefieldDF2second}\end{equation}
Since ${x}_{w, k}^{\Theta_{m} ({\bf b}, {\bf f}_k)} \in
K_{2m} ({\mathcal O}_{F_{k}, S}),$ then its image in
$K_{2m}^{et} ({\mathcal O}_{F_k, S};\, \Z/l^k)$ factors through
$H^2 ({\mathcal O}_{F_k, S}; \, \Z/l^k(m + 1)).$
Hence the proof is similar to that of
Lemma \ref{Lemma 4.3} with the use of
the projection formula for the negative twist in \'etale cohomology.
\end{proof}

\begin{proposition}\label{Proposition 7.7}
For every $k \geq k(v)$, we have
$$\partial_{F} (\Lambda^{\prime} (\xi_{v, k}; l^k)) =
(N(\xi_{w, k} \ast \beta_{k}^{\ast \, n-m}))^{\Theta_{m} ({\bf b}, {\bf f})}$$
\end{proposition}

\begin{proof}
The proof is similar to that of Proposition \ref{Proposition 4.4}.
The diagram at the end of section 3 gives the following
commutative  diagram of $K$--groups with coefficients

$$\xymatrix{
K_{2n} ({\mathcal O}_{E, S} ;\, \Z/l^k ) \quad
\ar@<0.1ex>[d]^{Tr_{E/F}} \ar[r]^{\partial_{E}} \quad &
\quad\
\bigoplus_{v \in S} \bigoplus_{w\, | \,v} K_{2n-1} (k_w ; \, \Z/l^k )
\ar@<0.1ex>[d]^{N}\\
K_{2n} ({\mathcal O}_{F, S}  ;\, \Z/l^k) \quad
\ar[r]^{\partial_{F}} &
\quad \bigoplus_{v \in S} \, K_{2n-1} (k_v, ;\, \Z/l^k )}\,,
\label{diagram 2.6second}$$
where $N := \bigoplus_{v} \bigoplus_{w\, | \,v} N_{w/v}.$
Hence we have $\partial_{F} \circ Tr_{E/F} = N \circ \partial_{E}.$
The compatibilities of some of the natural maps mentioned in section 3
which will be used next can be expressed via the following commutative diagrams,
explaining the action  of the groups $G(E/K)$ and $G(F/K)$ on the $K$--groups with coefficients
in the diagram above.
For $j > 0$ we use the following comutative diagram.
$$\xymatrix{
K_{2j} ({\mathcal O}_{E, S}; \, \Z/l^k ) \quad \ar@<0.1ex>[d]^{\sigma_{{\bf a}}^{-1}}
\ar[r]^{r_w} & \quad
K_{2j} (k_{w}; \, \Z/l^k)
\ar@<0.1ex>[d]^{\sigma_{{\bf a}}^{-1}}\\
K_{2j} ({\mathcal O}_{E, S}; \, \Z/l^k ) \quad
\ar[r]^{r_{w^{\sigma_{{\bf a}}^{-1}}}} & \quad
\, K_{2j} (k_{w^{\sigma_{{\bf a}}^{-1}}}; \, \Z/l^k)}
\label{diagram 7.7}$$
The above diagram gives the following equality:
\begin{equation}\label{reduction1}  r_{w^{\sigma_{{\bf a}}^{-1}}}(\beta_{k}^{\ast \, n-m}) =
r_{w^{\sigma_{{\bf a}}^{-1}}}((\beta_{k}^{\ast \, n-m})^{ N {\bf a}^{n-m}\sigma_{{\bf a}}^{-1}})
= (r_w (\beta_{k}^{\ast \, n-m}))^{ N {\bf a}^{n-m} \sigma_{{\bf a}}^{-1}}.\end{equation}
For any $j \in \Z$, we have the following commutative diagram:
$$\xymatrix{
H^0 ({\mathcal O}_{E, S}; \, \Z/l^k (j)) \quad \ar@<0.1ex>[d]^{\sigma_{{\bf a}}^{-1}}
\ar[r]^{r_w} & \quad
H^0 (k_{w}; \, \Z/l^k (j))
\ar@<0.1ex>[d]^{\sigma_{{\bf a}}^{-1}}\\
H^{0} ({\mathcal O}_{E, S}; \, \Z/l^k (j) ) \quad
\ar[r]^{r_{w^{\sigma_{{\bf a}}^{-1}}}} & \quad
\, H^0 (k_{w^{\sigma_{{\bf a}}^{-1}}}; \, \Z/l^k (j))}
\label{diagram 7.71}$$
If $\xi_{l^k} := exp (\frac{2 \, \pi \, i}{l^k})$ is the generator of $\mu_{l^k}$ then the
above diagram gives
\begin{equation}\label{reduction2}  r_{w^{\sigma_{{\bf a}}^{-1}}}(\xi_{l^k}^{\otimes \, n-m}) =
r_{w^{\sigma_{{\bf a}}^{-1}}}(\xi_{l^k}^{\otimes \, n-m})^{ N {\bf a}^{n-m}\sigma_{{\bf a}}^{-1}})
= (r_w (\xi_{l^k}^{\otimes \, n-m}))^{ N {\bf a}^{n-m} \sigma_{{\bf a}}^{-1}}.\end{equation}

We can write the $m$--th Stickelberger element as follows
\begin{equation}\label{theta-m} \Theta_{m} ({\bf b}, {\bf f}_k) =
{\sum_{{\bf a} \, {\rm{mod}} \, {\bf f}_k}}' \,\, (\sum_{{\bf c} \, {\rm{mod}} \, {\bf f}_k, \,
w^{\sigma_{{\bf c}^{-1}}} = w}\,
\Delta_{m+1} ({\bf a} {\bf c}, {\bf b}, {\bf f}) \sigma_{{\bf c}^{-1}}) \sigma_{{\bf a}^{-1}}\,,\end{equation}
where ${\sum'_{{\bf a} \, \rm{mod} \, {\bf f}_k}}$ denotes the sum over a maximal set $\mathcal S$
of ideal classes ${\bf a} \mod {\bf f}_k$,  such that the primes $w^{\sigma_{{\bf a}}^{-1}}$, for ${\bf a}\in\mathcal S$,  are
distinct. By formula (\ref{deltacong1}), for every $m \geq 1$ and $n \geq 1$ we have
$$ \Delta_{n+1} ({\bf a}, {\bf b}, {\bf f}) \equiv
N {\bf a}^{n- m} \, N {\bf b}^{n- m} \, \Delta_{m+1} ({\bf a} {\bf c}, {\bf b}, {\bf f})
\mod \,\, w_{\text{min} \, \{m, n \}}(K_{{\bf f}})$$
(see \cite{DR}).
It is clear that for all $m \geq 1$ and $n \geq 1$ we get the following
congruence $\mod
w_{\text{min} \, \{m, n \}}(K_{{\bf f}_k}).$

$$\Theta_{n} ({\bf b}, {\bf f}_k)\equiv$$
$${\sum_{{\bf a} \, {\rm{mod}} \, {\bf f}_k}}' \,\, (\sum_{{\bf c} \, {\rm{mod}} \, {\bf f}_k, \,
w^{\sigma_{{\bf c}^{-1}}} = w}\,
N {\bf a}^{n-m} \, N {\bf c}^{n-m}\, N {\bf b}^{n-m} \, \Delta_{m+1} ({\bf a} {\bf c}, {\bf b},
{\bf f}_k)
\sigma_{{\bf c}^{-1}}) \sigma_{{\bf a}^{-1}}$$

\noindent Equalities (\ref{reduction1}), (\ref{reduction2}), (\ref{theta-m}), Lemma
\ref{Lemma 7.2}, the result of Gillet \cite{Gi}, the compatibility of
$t (n-m)$ and $t_v (n-m)$ with $\partial$ and
the above congruences satisfied by Stickelberger elements lead in both cases $n- m \geq 0$
and $n - m < 0$ to the following equalities.

\begin{eqnarray}\nonumber
\partial_{E} ( x_{w, k}^{\Theta_{m} ({\bf b}, {\bf f}_k)} \ast
\beta_{k}^{\ast \, n-m})^{N {\bf b}^{n-m}}= \\
 \nonumber
 ={\sum_{{\bf a} \, {\rm{mod}} \, {\bf f}_k}}'
\,\, \xi_{w, k}^{\sum_{{\bf c} \, {\rm{mod}} \, {\bf f}_k, \, w^{\sigma_{{\bf c}^{-1}}} = w} \,
\Delta_{m+1} ({\bf a} {\bf c}, {\bf b}, {\bf f}_k) \sigma_{({\bf a} {\bf c})^{-1}}}
\ast (\beta_{k}^{{\ast \, n-m}})^{(N {\bf a} {\bf c})^{n-m} N {\bf b}^{n-m}
\sigma_{(ac)^{-1}}}= \\
\nonumber
 =(\xi_{w, k}
\ast \beta_{k}^{{\ast \, n-m}})^{
\sum'_{{\bf a} \, {\rm{mod}} \, {\bf f}_k} \sum_{{\bf c} \, {\rm{mod}} \, {\bf f}_k, \,
w^{\sigma_{{\bf c}^{-1}}} = w} \, \Delta_{m+1} ({\bf a} {\bf c}, {\bf b}, {\bf f}_k)
(N {\bf a} {\bf c})^{n-m} N {\bf b}^{n-m} \sigma_{({\bf a} {\bf c})^{-1}}} = \\
\nonumber
=(\xi_{w, k} \ast \beta_{k}^{{\ast \, n-m}})^{\Theta_{n} ({\bf b}, {\bf f}_k)}.
\end{eqnarray}
%$$ =
%{\sum_{a \, {\rm{mod}} \, f_k}}'
%\,\, (\xi_{w, k}
%\ast \beta_{k}^{{\ast \, n-1}})^{\sum_{c \, {\rm{mod}} \, f_k, \,
%w^{\sigma_{c^{-1}}} = w} \, \Delta_{2} (a c, b, f) (Nac)^{n-1} Nb^{n-1}
%\sigma_{(ac)^{-1}}} \,\, = $$
We finish the proof in the same way as that of Proposition \ref{Proposition 4.4}, by
applying the first commutative diagram and the equalities above:
$$ \partial_{F} (\Lambda^{\prime} (\xi_{v, k}; l^k)) =
N (\partial_{E} ( x_{w, k}^{\Theta_{m} ({\bf b}, {\bf f}_k)} \ast
\beta_{k}^{\ast \, n-m})^{N{\bf b}^{n-m}}) =
N ((\xi_{w, k} \ast \beta_{k}^{{\ast \, n-m}})^{\Theta_{n} ({\bf b}, {\bf f}_k)}) = $$
$$ = (N (\xi_{w, k} \ast \beta_{k}^{{\ast \, n-m}}))^{\Theta_{n} ({\bf b}, {\bf f})}.$$
\end{proof}

\noindent We define $\Lambda^{\prime} (\xi_{v})
\in K_{2n} ({\mathcal O}_{F, S})_l$
to be the element corresponding to
$$(\Lambda^{\prime} (\xi_{v, k}; l^k))_{k} \in
\varprojlim_{k} K_{2n} ({\mathcal O}_{F, S} ;\, \Z/l^k)$$
via the isomorphism (\ref{invlimringofintegers}). Also, we
define $\xi_v \in K_{2n-1} (k_v)_l$
to be the element corresponding to
$$(N(\xi_{w, k} \ast \beta_{k}^{\ast \, n-m}))_k \in
\varprojlim_{k} K_{2n-1} (k_v ;\, \Z/l^k)$$ via the
isomorphism (\ref{invlimfinitefields}).

\begin{proposition}\label{Proposition 7.8} For every $v$ such that
$l\mid q_{v}^n - 1$ and for all $k \geq k(v),$ there are homomorphisms
$$ \Lambda_{v, \, l^k}^{\prime} \, : \, K_{2n-1} (k_v ;\, \Z/l^k) \rightarrow
K_{2n} ({\mathcal O}_{F, S} ;\, \Z/l^k)$$
satisfying the equality
$$\Lambda_{v, \, l^k}^{\prime} \, (N(\xi_{w, k} \ast \beta_{k}^{\ast \, n-m}))
\,\, = \,\, \Lambda^{\prime} (\xi_{v, k}; l^k).$$
\end{proposition}

\begin{proof}
The definition of $\Lambda_{m}^{\prime}$
combined with the natural
isomorphism $K_{2m-1}(k_w) / l^k \cong K_{2m-1}(k_w; \Z/l^k\Z)$ and the natural
monomorphism
$$K_{2m-1}({\mathcal O}_{E, S}) / l^k \to K_{2m-1}({\mathcal O}_{E, S}; \Z/l^k\Z),$$
coming from the corresponding Bockstein exact sequences, leads to the following homomorphism
$$\Lambda_{m}^{\prime} : K_{2m-1}(k_w; \Z/l^k\Z) \to  K_{2m}({\mathcal O}_{E,S}; \Z/l^k\Z).$$
Multiplying on the target and on the source of this homomorphism with the $n-m$ power
of the Bott element if $n-m \geq 0$ (resp. applying the map $t_w (n-m)$ to the source and
$t (n-m)$ to the target if $n -m < 0$) under the observation that the following map is an
isomorphism:
$$\xymatrix
{K_{2m-1}(k_w; \Z/l^k\Z)\ar[r]_{\sim}^{{\ast\beta_k^{\ast n-m}}} &K_{2n-1}(k_w; \Z/l^k\Z)}$$
(cf. the notation of $t (j)$ and $t_w (j)$ )
show that there exists a unique homomorphism
$$\Lambda_{m}^{\prime} \ast \beta_{k}^{\ast \, n-m}\, :\,K_{2n-1} (k_w ;\, \Z/l^k) \rightarrow
K_{2n} ({\mathcal O}_{E, S} ;\, \Z/l^k),$$
sending  $ \xi_{w, k} \ast \beta_{k}^{\ast \, n-m} \rightarrow
x_{w, k}^{\Theta_{m} ({\bf b}, {\bf f}_{k})} \ast \beta_{k}^{\ast \, n-m}$.
Next, we compose the homomorphisms $\Lambda_{m}^{\prime} \ast \beta_{k}^{\ast \, n-m}$
defined above and
$$Tr_{E/F}\, : \, K_{2n} ({\mathcal O}_{E, S} ;\, \Z/l^k) \rightarrow
K_{2n} ({\mathcal O}_{F, S} ;\, \Z/l^k)$$
to obtain the following homomorphism:
$$ Tr_{E/F}\circ (\Lambda_{m}^{\prime} \ast \beta_{k}^{\ast \, n-m}):
K_{2n-1} (k_w ;\, \Z/l^k) \rightarrow
K_{2n} ({\mathcal O}_{F, S} ;\, \Z/l^k)\,.$$
We observe that this homomorphism factors through the quotient of $G(k_w/k_v)$--coinvariants
$$K_{2n-1} (k_w ;\, \Z/l^k)_{G(k_{w} / k_{v})}:=
K_{2n-1} (k_w ;\, \Z/l^k) / K_{2n-1} (k_w ;\, \Z/l^k)^{Fr_v - Id}\,,$$
and arguing in the same way as in the Proposition
\ref{Proposition 4.5} we get the homomorphism:
\begin{equation}
\xymatrix{\Lambda_{v, \, l^k}^{\prime}
\, : \, K_{2n-1} (k_v ; \, \Z/l^k) \ar[r] & K_{2n} ({{\mathcal O}_{F, S}} ; \, \Z/l^k) }
\label{tag 4.6}
\end{equation}
defined by
$$\Lambda_{v, \, l^k}^{\prime} (x):=
[Tr_{E/F}\circ (\Lambda_{1} \ast \beta_{k}^{\ast \, n-m})\circ N^{-1}(x)]^{N{\bf b}^{n-m}}\,,$$
for all $x\in K_{2n-1} (k_v ; \, \Z/l^k)$.
By definition, this map sends $N(\xi_{w, k} \ast \beta_{k}^{\ast \, n-m})$ onto the element
$\Lambda^{\prime} (\xi_{v, k}; l^k):=
Tr_{E/F}(x_{w, k}^{\Theta_{1} ({\bf b}, {\bf f}_{k})}
\ast \beta_{k}^{\ast \, n-m})^{N{\bf b}^{n-m}}$.
\end{proof}
\bigskip

\noindent Since the homomorphisms $\Lambda_{v, \, l^k}^{\prime}$ are compatible with the
coefficient
reduction maps $r_{k^{\prime}/k}$, for all $k'\geq k\geq k(v),$ we can construct homomorphisms
$$ \Lambda_{v}^{\prime} := \varprojlim_{k} \Lambda_{v, \, l^k}^{\prime}:
K_{2n-1} (k_v)_l \rightarrow K_{2n} ({\mathcal O}_{F, S})_l,
$$
for all $v$, satisfying assumptions of the last proposition.
Since $K_{2n}({\mathcal O}_{F, S}) \subset K_{2n} (F)$,
we get the maps:
$$\Lambda_{v}^{\prime} \, : \, K_{2n-1} (k_v)_l \rightarrow K_{2n} (F)_l.$$

\begin{definition}\label{Definition 7.9}
We define the map $\Lambda^{\prime}$ by
$$\Lambda^{\prime} \, : \, \bigoplus_{v} K_{2n-1} (k_v)_l
\, \rightarrow K_{2n} (F)_l$$
$$\Lambda^{\prime} :=\Lambda_{n}^{\prime} \, := \, \prod_{v} \Lambda_{v}^{\prime}.$$
\end{definition}

\begin{theorem}\label{Theorem 7.10}
The map $\Lambda^{\prime} := \Lambda_{n}^{\prime}$ satisfies the following property.
$$\partial_{F} \circ \Lambda^{\prime} (\xi_{v}) =
\xi_{v}^{\Theta_{n} ({\bf b}, {\bf f})}
$$
\end{theorem}
\begin{proof}
The theorem follows by Propositions \ref{Proposition 7.7} and
\ref{Proposition 7.8} (cf. the proof of Theorem \ref{Theorem 5.1}).
\end{proof}
\begin{theorem}\label{Theorem 7.11}
Let $m > 0$ be a natural number. Assume that the Stickelberger elements
$\Theta_{m} ({\bf b}, {\bf f}_k)$ annihilate
the groups $K_{2 m} ({\mathcal O}_{F_k})_l$ for all $k \geq 1.$
Then the Stickelberger element $\Theta_{n} ({\bf b}, {\bf f})$ annihilates the group
$div \, K_{2n} (F)_l$ for every $n \geq 1.$
\end{theorem}
\begin{proof}
The proof is very similar to the proof of Theorem \ref{Theorem 5.4}.
\end{proof}
\begin{corollary}\label{Corollary 7.12}
Let $m > 0$ be a natural number. Assume that the Stickelberger elements
$\Theta_{m} ({\bf b}, {\bf f}_k)$ annihilate
the groups $K_{2 m} ({\mathcal O}_{F_k})_l$ for all $k \geq 1.$
Then the Stickelberger element $\Theta_{n} ({\bf b}, {\bf f}_k)$ annihilates the group
$div \, K_{2n} (F_k)_l$ for every $k \geq 0$ and every $n \geq 1.$
In particular $\Theta_{n} ({\bf b}, {\bf f}_{\infty})$ annihilates the group
$\varinjlim_{k} div \, K_{2n} (F_{k})_l$ for every $n \geq 1.$
\end{corollary}
\begin{proof}
Follows immediately from Theorem \ref{Theorem 7.11}.
\end{proof}
\bigskip

We can easily establish the \'etale $K$--theoretic analogues of these constructions.
Assume that the Stickelberger elements
$\Theta_{m} ({\bf b}, {\bf f}_k)$ annihilate
the groups $K_{2 m}^{et} ({\mathcal O}_{F_k})_l$ for all $k \geq 1.$
In analogy with \S6, we construct for every $n > 0$ the following
map $$\Lambda^{\prime \, et} \, : \, \bigoplus_{v} K_{2n-1}^{et} (k_v)
\, \rightarrow K_{2n}^{et} (F)_l.$$ This is the \'etale analogue
of our map $\Lambda^{\prime}.$ This construction leads immediately to
the following results.

\begin{theorem}\label{Theorem 7.13}
The map $\Lambda^{\prime \, et}$ satisfies the following property.
$$\partial_{F}^{et} \circ \Lambda^{et} (\zeta_{v}) =
\zeta_{v}^{\Theta_{n} ({\bf b}, {\bf f})}
$$
\end{theorem}

\begin{theorem}\label{Theorem 7.14}
Assume that the Stickelberger elements $\Theta_{m} ({\bf b}, {\bf f}_k)$ annihilate
the groups $K_{2m}^{et} ({\mathcal O}_{F_k})_l$ for all $k \geq 1.$
Then the Stickelberger element $\Theta_{n} ({\bf b}, {\bf f})$ annihilates the group
$div \, K_{2n}^{et} (F)_{l}$ for all $n \geq 1
.$
\end{theorem}

%\noindent
%\section{Euler systems based on $\Lambda$ elements}

%Consider now a CM field $F_L /F$ such that the conductor of
%$F/K$ is ${\bf f}$ and conductor of $F_{\bf L}/K$ is ${\bf f} {\bf L}$
%and ${\bf L} = {\bf l}_1 \dots {\bf l}_t$ is a product of different
%prime ideals of ${\mathcal O}_K$ such that ${\bf l} \, \not| \, {\bf L}.$
%Put $\gamma_{\bf l} := 1 + N({\bf l})^n \sigma_{{\bf l}}^{-1} +
%N({\bf l})^{2n} \sigma_{{\bf l}}^{-2} + \dots$
%Similarly as before we can define,
%for $k \geq k (v)$ and $E := F(\mu_{l^k}),$ elements

%$$\Lambda_n ({\bf L}, \xi_{v, k}; l^k) :=
%Tr_{E/F} ( x_{w, k}^{\Theta_{1} ({\bf b}_{\bf L}, {\bf f L}_k) } \ast
%\beta_{k}^{\ast \, n-1})^{N {\bf b}_{\bf L}^{n-1}} \in
%K_{2n} ({\mathcal O}_{F_{{\bf L}}, S};\, \Z/l^k).$$

%\begin{theorem}\label{Theorem 7.1}
%For every $k^{\prime} \geq k \geq k(v)$ we have
%$$\partial_{F_{{\bf L}}} (\Lambda_n ({\bf L}, \xi_{v, k}; l^k)) =
%(N(\xi_{w, k} \ast \beta_{k}^{\ast \, n-1}))^{\Theta_{n} ({\bf b_{\bf L}}, {\bf f L})}$$

%$$Tr_{F_{{\bf L}}/F_{{\bf L}/{{\bf l}_i}}} (\Lambda_n ({\bf L}, \xi_{v, k}; l^k)) =
%(\Lambda_n ({\bf L}/{{\bf l}_i}, \xi_{v_{i}, k}; l^k))^{1 - N({\bf l}_i) Fr^{-1}_{{\bf l}_i}}$$
%\end{theorem}

%\begin{proof}

%\end{proof}

%to be continued.............

%\medskip

%\noindent

%\begin{example}\label{Example 2.5}

%\end{example}

\medskip

\noindent
{\it Acknowledgments}:\quad
The first author would like to thank the University
of California, San Diego for hospitality and financial
support during his stay in April 2009, when this collaboration began. Also, he
thanks the
SFB in Muenster for hospitality and financial
support during his stay in September 2009 and the Max Planck Institute in Bonn
for hospitality and financial support during his stay in April and May 2010.
\bigskip

%-----------------BIBLIOGRAPHY----------

\bibliographystyle{amsplain}

\end{document}